\documentclass[oneside]{amsart}

\usepackage{graphicx} 
\usepackage[utf8]{inputenc} 
\usepackage{xcolor}
\usepackage{amsmath}
\usepackage{geometry}
\usepackage{hyperref} 
\usepackage{cleveref} 
\usepackage{algorithm}
\usepackage{algorithmic}

\crefname{equation}{}{}
\crefname{theorem}{Theorem}{Theorems}
\crefname{lemma}{Lemma}{Lemmas}
\crefname{corollary}{Corollary}{Corollaries}
\crefname{proposition}{Proposition}{Propositions}
\crefname{remark}{Remark}{Remarks}
\crefname{section}{Section}{Sections}
\crefname{figure}{Figure}{Figures}
\crefname{algorithm}{Algorithm}{Algorithms}
\crefname{definition}{Definition}{Definitions}

\usepackage{cite}	

\usepackage{amsfonts}
\usepackage{amssymb}
\usepackage{placeins}
\usepackage{pgfplots}
\pgfplotsset{compat=1.18}

\numberwithin{equation}{section}

\newtheorem{theorem}{Theorem}[section]

\newtheorem{proposition}[theorem]{Proposition}
\theoremstyle{definition}

\theoremstyle{remark}
\newtheorem{remark}{Remark}[section]

\DeclareMathOperator{\rank}{rank}

\DeclareMathOperator{\blockspan}{blockspan}

\makeatletter
\renewcommand{\paragraph}[1]{\medskip\noindent\textit{#1.}\quad}
\makeatother

\newcommand{\R}{\mathbb{R}} 
\newcommand{\C}{\mathbb{C}} 
\newcommand{\T}{\mathbb{T}} 
\renewcommand{\vec}[1]{\boldsymbol{#1}}		
\newcommand{\malpha}{\boldsymbol{\alpha}}
\newcommand{\mrho}{\boldsymbol{\rho}}
\newcommand{\de}{\mathrm{d}} 
\newcommand{\ttilde}[1]{\widetilde{#1}}

\newcommand{\K}{\mathcal{K}} 
\newcommand{\kryl}{\K}
\newcommand{\bkryl}{\kryl^{\square}}

\newcommand{\ebkryl}{\mathcal{E}^{\square}}
\newcommand{\rbkryl}{\mathcal{Q}^{\square}}
\newcommand{\btheta}{\boldsymbol{\theta}}
\newcommand{\bmu}{\boldsymbol{\mu}}

\newcommand{\poly}{\mathbb{P}}
\newcommand{\laur}{\mathbb{L}}
\renewcommand{\b}{\vec{b}}
\newcommand{\B}{\vec{B}}
\newcommand{\X}{\vec{X}}
\newcommand{\Y}{\vec{Y}}
\newcommand{\V}{\vec{V}}
\newcommand{\W}{\vec{W}}

\newcommand{\llangle}{\langle\!\langle}
\newcommand{\rrangle}{\rangle\!\rangle}

\title[Block Krylov subspaces and orthogonal matrix polynomials]{Block Krylov subspaces and orthogonal matrix polynomials: a structural correspondence with applications to unitary matrices}
\author[M. Rinelli]{Michele Rinelli$^*$}
\author[R. Vandebril]{Raf Vandebril$^*$}
\thanks{$^*$Numerical Analysis and Applied Mathematics (NUMA), Department of Computer Science, KU Leuven, Leuven, Belgium. \\
Michele Rinelli: \texttt{michele.rinelli@kuleuven.be},
Raf Vandebril: \texttt{raf.vandebril@kuleuven.be}}

\begin{document}

\begin{abstract}
  We study the connection between block Krylov subspaces and matrix orthogonal functions. 
  Under a no-deflation assumption, we show that polynomial block Krylov subspaces are isometrically isomorphic to spaces of matrix polynomials of bounded degree, providing a unified framework for the analysis and construction of orthonormal bases and recurrence relations. The same correspondence holds for rational block Krylov subspaces and matrix-valued rational functions, and in the extended Krylov setting this leads naturally to Laurent matrix polynomials. 
  When the matrix $A$ is normal, we prove that the induced inner product admits a representation in terms of a discrete spectral matrix measure, extending a classical result for Hermitian matrices. 
  In the unitary case, where the measure is supported on the unit circle, this connection allows us to transfer the Szeg\H{o} recurrence for orthogonal matrix polynomials and the CMV framework for Laurent matrix polynomials to the block Krylov setting, yielding efficient procedures for the orthogonalization of polynomial and extended block Krylov subspaces.
\end{abstract}

\maketitle

\section{Introduction}

The interplay between Krylov subspace methods and orthogonal polynomials is classical in numerical linear algebra. For a Hermitian matrix $A$ and a starting vector $\b$, the Lanczos process constructs an orthonormal basis of the Krylov subspace $\kryl_m(A,\b)$ through a three-term recurrence, and the corresponding projected matrix is tridiagonal. On the other hand, orthogonal polynomials on the real line, defined with respect to an integral inner product, satisfy a three-term recurrence governed by a Jacobi matrix. This analogy is not merely formal: the Lanczos vectors may be interpreted as orthogonal polynomials in~$A$ acting on the starting vector, a viewpoint that lies behind the classical links with moments, quadrature, and error estimates; see Meurant~\cite{MeurantLanczos06}, and the books by Golub and Meurant~\cite{GoMe10}, and Liesen and Strakos~\cite{LieStraKrylov}.

A closely related picture emerges for unitary matrices. In this setting, the Arnoldi process admits short recurrences that reflect the unitary structure of the matrix. In particular, Gragg's isometric Arnoldi algorithm~\cite{Gragg93isometric} generates an orthonormal Krylov basis in exact analogy with the Szeg\H{o} recurrence for orthogonal polynomials on the unit circle~\cite{SzegoBook75}. This parallel was made especially transparent in the survey by Watkins~\cite{Watkins93}, where the Hermitian/real-line and unitary/unit-circle settings are placed side by side within a common eigenvalue-oriented perspective.

 Block variants of Krylov methods arise naturally in several applications, such as linear systems with multiple right-hand sides~\cite{OLeary80,Simoncini95}, the matrix-vector multiplication of matrix functions on several vectors~\cite{FrLuSz17,FrLuSz20}, Lyapunov and Sylvester equations with low-rank right-hand sides~\cite{blockKrylovSylvester00}, and model order reduction for multi-input multi-output systems~\cite{Antoulas05,ZimDruSim25}. In many of these settings one may also run independent single-vector Krylov methods. The main computational appeal of block methods is that products of the form $A[\vec x_1,\dots,\vec x_s]$ can be performed in a single matrix-block multiplication, which is typically preferable on modern architectures to computing $A\vec{x}_1,\dots,A\vec{x}_s$ separately; see Carson, Lund and Rozlo\v{z}n\'ik~\cite{CarsonLundetal22}, and the thesis of Lund~\cite{LundPhD}.  
 For eigenvalue computations, block methods also offer a concrete advantage over independent single-vector runs, since they can capture several directions associated with multiple or clustered eigenvalues~\cite{Parlett98}.

 Whereas orthogonal polynomials play a prominent role in numerical linear algebra, orthogonal matrix polynomials have so far received much less attention. On the theoretical side, however, their structure is well developed: for matrix-valued measures on the real line one obtains orthogonal matrix polynomials satisfying matrix three-term recurrence relations, while for matrix-valued measures on the unit circle one obtains matrix Szeg\H{o} recurrences in terms of matrix Verblunsky coefficients; see, for instance, Sinap and Van Assche~\cite{SinVanAs94}, and the survey by Damanik, Pushnitski, and Simon~\cite{DaPuSi08}. In the unit-circle setting one may also work with orthogonal Laurent matrix polynomials and the associated CMV basis, whose multiplication operator is represented by a sparse block CMV matrix~\cite{SimonCMV07}.

In this paper we present a unified framework linking block Krylov subspaces and orthogonal matrix polynomials. This connection is already mentioned by Golub and Meurant~\cite{GoMe10} and Lund~\cite{LundPhD} in the context of symmetric block Gauss quadrature. Here we develop it in a general setting. Under a no-deflation assumption, we show that block Krylov subspaces are naturally identified with spaces of matrix polynomials endowed with a suitable matrix-valued inner product. We then extend the same perspective to block rational Krylov subspaces and orthogonal matrix-valued rational functions; in the extended block Krylov setting, these reduce to Laurent matrix polynomials. The case of deflating subspaces, treated in detail by Gutknecht and Schmelzer~\cite{BlockGradeGutknecht09}, is also discussed in connection with rectangular matrix polynomials. 

To illustrate the strength of this framework, we specialize to the unitary case. The Szeg\H{o} recurrence for orthogonal matrix polynomials on the unit circle translates into short recurrences for the orthogonal block vectors spanning the block Krylov subspace, yielding a block analogue of the isometric Arnoldi process. We also exhibit the structured rank of the projected matrix through a Schur parametrization that depends on the recurrence parameters. Likewise, the CMV framework, based on orthogonal Laurent matrix polynomials and a sparse recurrence matrix, leads to an efficient orthogonalization procedure for extended block Krylov subspaces.

The paper is organized as follows. In~\cref{sec:preliminaries} we revisit the classical scalar link between Krylov subspaces and orthogonal polynomials and collect the preliminaries on block Krylov subspaces and orthogonal matrix polynomials. In~\cref{sec:block-krylov-orthogonal-matrix-polynomials-correspondence} we establish the correspondence between block Krylov subspaces and orthogonal matrix polynomials, showing in particular that, for normal matrices, the induced matrix-valued inner product is associated with a spectral measure supported on $\sigma(A)$. We then extend the construction to the rational setting and, in particular, to extended block Krylov subspaces and Laurent matrix polynomials. In~\cref{sec:unitary-block-Krylov} we exploit the unitary case and the theory of orthogonal matrix polynomials on the unit circle to derive the block isometric Arnoldi algorithm. In~\cref{sec:unitary-extended-block-krylov-CMV} we use the sparsity pattern of the block CMV matrix to obtain an efficient orthogonalization procedure for extended block Krylov subspaces. The numerical experiments in~\cref{sec:numerical-experiments} assess the efficiency and accuracy of the methods.

\section{Preliminaries}
\label{sec:preliminaries}
This section recalls standard material on Krylov subspaces, block Krylov methods, and orthogonal matrix polynomials. 

\subsection{Notation}
Block vectors in $\C^{n\times s}$ are denoted by capital bold letters such as $\X$. Single brackets $\langle\cdot,\cdot\rangle$ are used to denote scalar inner products, while double brackets $\llangle\cdot,\cdot\rrangle$ denote matrix-valued inner products. We use $\delta_{i,j}$ for the Kronecker delta. The identity matrix of size $s\times s$ is denoted by $I_s$, and simply by $I$ when its size is clear from the context. Likewise, the zero matrix is denoted by $\boldsymbol{0}$ when its size is clear, and by $\vec 0_{m,n}$ when we wish to specify the size $m\times n$. The conjugate transpose of a matrix $A=(a_{i,j})_{i,j}$ is denoted by $A^*=(\overline{a}_{j,i})_{i,j}$. Finally, $\|A\|_2$ and $\|A\|_F$ denote the spectral norm and the Frobenius norm, respectively.

\subsection{Classical link between Krylov subspaces and orthogonal polynomials}

Let $A\in\C^{n\times n}$ and $\vec b\in\C^n$.
The Krylov subspace of order $m$ associated with $(A,\vec b)$ is defined as
\begin{equation*}
\kryl_{m}(A,\vec b)
:= \operatorname{span}\{\vec b, A\vec b, \dots, A^{m-1}\vec b\}.
\end{equation*}
Equivalently, every $\vec v\in\mathcal K_{m}(A,\vec b)$ can be written as
$\vec v=p(A)\vec b$, where $p$ is a scalar polynomial with $\deg p<m$.

Since the vectors $\vec b, A\vec b,\dots,A^{m-1}\vec b$ typically form an
ill-conditioned basis, many Krylov methods construct an orthonormal basis of
$\mathcal K_{m}(A,\vec b)$ with respect to the Euclidean inner product
$\langle \vec v,\vec w\rangle=\vec v^*\vec w$.
The most common procedure is the Arnoldi algorithm, which generates
orthonormal vectors $\vec v_1,\dots,\vec v_{m}$ satisfying the recurrence
\begin{equation*}
A\vec v_k = \sum_{j=1}^{k+1} h_{j,k}\vec v_j,
\qquad
h_{j,k}=\vec v_j^*A\vec v_k.
\end{equation*}
By setting $V_{m}=[\vec v_1,\dots,\vec v_{m}]$ and denoting by
$H_{m}$ and $\underline{H}_{m}$ the associated upper Hessenberg
matrices, this recurrence can be equivalently written as
\begin{equation*}
A V_{m}
= V_{m} H_{m}
+ h_{m+1,m}\vec v_{m+1}\vec e_{m}^*
= V_{m+1}\underline{H}_{m}.
\end{equation*}

The pair $(A,\vec b)$ induces a polynomial inner product defined by
\begin{equation}
  \label{eq:poly-inner-product-scalar}
\langle p,q\rangle_{A,\vec b}
:= \vec b^* p(A)^* q(A)\vec b.
\end{equation}
When restricted to the space of polynomials of degree at most $m-1$, which we denote by $\poly_{m-1}$, the inner product \cref{eq:poly-inner-product-scalar} is nondegenerate if and only if
$\dim \mathcal K_{m}(A,\vec b)=m$, which is equivalent to the absence of
breakdown in the Arnoldi process up to step $m$.
Under this assumption, the mapping
$p\mapsto p(A)\vec b$ defines an isometry between
$\poly_{m-1}$ and $\mathcal K_{m}(A,\vec b)$.
As a consequence, orthonormal Krylov bases correspond to families of orthogonal
polynomials with respect to $\langle\cdot,\cdot\rangle_{A,\vec b}$, and the
Arnoldi recurrence contains the associated polynomial recurrence.

If $A$ is normal and
$A=\sum_{j=1}^n \lambda_j \vec u_j \vec u_j^*$ is an orthonormal eigendecomposition,
the inner product $\langle\cdot,\cdot\rangle_{A,\vec b}$ admits the discrete
representation
\begin{equation}
  \label{eq:discrete-inner-product-scalar}
\langle p,q\rangle_{A,\vec b}
= \sum_{j=1}^n \overline{p(\lambda_j)}\,q(\lambda_j)\,
|\vec u_j^*\vec b|^2,
\end{equation}
that is, a discrete inner product for polynomials supported on the spectrum of $A$. A proof of \cref{eq:discrete-inner-product-scalar} can be found in \cite[Theorem 4.2]{GoMe10} for the real symmetric case, but it extends trivially to any normal matrix, since it only relies on a spectral decomposition of $A$.
In this case, the Hessenberg matrix generated by Arnoldi (or Lanczos, in the
Hermitian setting) represents the recurrence satisfied by the corresponding
orthonormal polynomials.

The discrete representation above admits a natural measure-theoretic
interpretation.
Indeed, we can rewrite the inner product as
\begin{equation*}
\langle p,q\rangle_{A,\vec b}
=
\int_{\sigma(A)}
\overline{p(z)}\,q(z)\,\mathrm d\mu(z),
\end{equation*}
where $\mu$ is the discrete measure
\begin{equation*}
\mu = \sum_{j=1}^n |w_j|^2 \,\delta_{\lambda_j},\quad w_j = \vec u_j^*\vec b,
\end{equation*}
where $\delta_{\lambda}$ is the Dirac delta measure concentrated at $\lambda$.
We call $\mu$ the \emph{spectral measure} of $(A,\vec b)$.

In the Hermitian case, spectral measures
play a central role in the theory of Lanczos method and Gaussian quadrature, since they provide the link between tridiagonal Lanczos recurrences and orthogonal
polynomials on the real line.
This viewpoint is classical in numerical linear algebra and approximation
theory, and appears in the analysis of moment problems, Gauss-type quadrature
rules, and matrix function approximation; see, e.g.,~\cite{GoMe10}.
In this paper, we extend this perspective to the block and
non-Hermitian normal settings, where the induced measure becomes
matrix-valued.

\subsection{Block Krylov subspaces}

Let $A\in\C^{n\times n}$ and let $\vec B\in\C^{n\times s}$ be a block vector.
The (classical) block Krylov subspace of order
$m$ and block size $s$ associated with $(A,\vec B)$ is
\begin{equation*}
\bkryl_{m}(A,\vec B)
:=
\left\{
\sum_{k=0}^{m-1} A^k \vec B C_k
:\;
C_k\in\C^{s\times s}
\right\}.
\end{equation*}
The sequence $\B, A\B,\dots,A^{m-1}\B$ does not span in the usual vector-space sense: elements of $\bkryl_m(A,\B)$ are formed using $s\times s$ matrix coefficients. To make this explicit, we use the notation
\begin{equation*}
  \blockspan\{\X_1,\dots,\X_{m}\} = \left\{  \sum_{k=1}^{m} \X_k C_k: C_k\in\C^{s\times s}\right\},
\end{equation*}
where $\X_1,\dots,\X_{m}$ are $n\times s$ block vectors. Thus, we can write
\begin{equation*}
  \bkryl_{m}(A,\B) = \blockspan\{ \B,A\B,\dots,A^{m-1}\B \}.
\end{equation*}

Block Krylov methods construct bases of $\bkryl_{m}(A,\vec B)$ that are
orthonormal with respect to the matrix-valued Euclidean inner product
\begin{equation*}
\llangle \vec X,\vec Y\rrangle := \vec X^*\vec Y.
\end{equation*}
A sequence $\vec V_1,\dots,\vec V_{m}$ is orthonormal if
\begin{equation*}
\vec V_i^*\vec V_j = \delta_{i,j} I_{s},
\qquad i,j=1,\dots,m,
\end{equation*}
and forms a basis if
\begin{equation*}
\bkryl_{m}(A,\vec B)
=
\blockspan\{ \V_1,\dots,\V_{m} \}.
\end{equation*}
Additionally, we say that the basis is \emph{nested} if 
\begin{equation*}
  \bkryl_k(A,\B) = \blockspan\{ \V_1,\dots,\V_k \},\quad \text{for all $k=1,\dots,m$.}
\end{equation*}
Such a basis can be generated by the block Arnoldi algorithm, which we report in~\cref{alg:block-arnoldi}; for more details see, e.g.,~\cite{FrLuSz17,Simoncini95}.
\begin{algorithm}[ht]
\caption{Block Arnoldi algorithm}
\label{alg:block-arnoldi}
\begin{algorithmic}[1]
\REQUIRE Matrix $A\in\C^{n\times n}$, starting block vector $\B\in\C^{n\times s}$ such that $\B^*\B=I_{s}$, number of steps $m$
\ENSURE Nested orthonormal basis $\vec V_1,\dots,\vec V_{m}$ of $\bkryl_{m}(A,\B)$, and recurrence coefficients $H_{i,j}$

\STATE \textbf{Initialize.} Set $\V_1 = \B$

\FOR{$k=1,2,\dots,m-1$}

    \STATE \textbf{Multiplication with $A$.} Compute $\vec X_{k+1} = A\V_k$
    \STATE \textbf{Recurrence coefficients.} Compute $H_{h,k} = \V_h^* \X_{k+1}$, $h=1,\dots,k$
    \STATE Compute $\W_{k+1} = \X_{k+1} - \sum_{h=1}^k \V_kH_{h,k}$
    \IF{$\rank(\W_{k+1}) < s$}
      \STATE Breakdown if $\rank(\W_k)=0$, deflation if $0<\rank(\W_{k+1})<s$\label{line:breakdown-deflation-block-arnoldi}
    \ELSE 
      \STATE Construct $\V_{k+1}$, $H_{k+1,k}$ from a thin QR factorization $\W_{k+1} = \V_{k+1} H_{k+1,k}$
    \ENDIF
\ENDFOR
\end{algorithmic}
\end{algorithm}

In the absence of breakdown or deflation in~\cref{alg:block-arnoldi}, the basis vectors satisfy the recurrence relation
\begin{equation*}
  \label{eq:block-Arnoldi-recurrence}
A\vec V_k
=
\sum_{h=1}^{k+1} \vec V_{h} H_{h,k},
\qquad k=1,\dots,m-1.
\end{equation*}
Equivalently, 
\begin{equation}
  \label{eq:recurrence-block-arnoldi}
A\mathcal V_{m-1}
=
\mathcal V_{m-1}\,\mathcal H_{m-1} + \V_{m}\vec E_{m-1}^TH_{m,m-1},
\end{equation}
where $\mathcal V_{m-1}=[\vec V_1,\dots,\vec V_{m-1}]$, $\vec E_{m-1}^T =[\vec{0},\dots,\vec{0},I_{s}]$,  
and $\mathcal H_{m-1}$ is block upper Hessenberg with $H_{h,k}$ as blocks. When $A$ is Hermitian,~\cref{alg:block-arnoldi} reduces to the block Lanczos algorithm, so the recurrence relation \cref{eq:block-Arnoldi-recurrence} includes only three terms, and the recurrence matrix $\mathcal{H}_{m-1}$ becomes block tridiagonal.

The following proposition characterizes the occurrence of breakdown in the block Arnoldi process; see~\cite{BlockGradeGutknecht09} for further details.
\begin{proposition}
  \label{prop:no-deflation-block-krylov}
  Let $\bkryl_m(A,\B)$ be the block Krylov subspace associated with $A\in\C^{n\times n}$ and $\B\in\C^{n\times s}$. The following statements are equivalent.
  \begin{enumerate}
    \item The block Arnoldi algorithm does not break down before step $m$.
    \item The space $\bkryl_m(A,\B)$ admits an orthonormal basis of block vectors.
    \item The dimension of $\bkryl_m(A,\B)$ as a $\C$-vector space is $s^2 m$.
  \end{enumerate}
\end{proposition}
If the equivalent conditions of~\cref{prop:no-deflation-block-krylov} are satisfied, we say that the block Krylov subspace $\bkryl_m(A,\B)$ is \emph{nondeflating}.

If $\W_{k+1}= \vec 0$ at step $k$, then~\cref{alg:block-arnoldi} breaks down. In this case,
\begin{equation}
  \label{eq:stationary-block-krylov-subspaces}
  \bkryl_k(A,\B) = \bkryl_{k+1}(A,\B) = \bkryl_j(A,\B),
\end{equation}
for all $j\geq k+1$. This corresponds to a \emph{lucky breakdown}, meaning that all the information contained in the higher-order block Krylov subspaces is already captured by $\bkryl_k(A,\B)$. For example, in the context of linear systems $A\vec X=\vec \B$, a breakdown implies that the exact solution belongs to $\bkryl_{k}(A,\B)$.

If, on the other hand, $\W_{k+1}\neq 0$ but $\rank(\W_{k+1})<s$, then~\cref{eq:stationary-block-krylov-subspaces} no longer holds, and the construction can proceed after a deflation. More precisely, the thin QR factorization of $\W_{k+1}$ takes the form
\[
  \W_{k+1} = \V_{k+1} H_{k+1,k},
\]
where $\V_{k+1}\in\C^{n\times r}$ is a reduced block vector, $H_{k+1,k}\in\C^{r\times s}$ is upper triangular, and $r=\rank(\W_{k+1})<s$.

The absence of deflation in $\bkryl_m(A,\B)$ will be essential in establishing the connection with orthogonal matrix polynomials. Nevertheless, we will later discuss possible ways to handle the deflated case.

\subsection{Action of matrix polynomials on block vectors}
In the vector case, polynomials play an important role for the description and analysis of Krylov subspace methods. In the block case, the analogous role is taken by matrix polynomials. 

A matrix polynomial of degree at most $m-1$ is an expression
\begin{equation*}
P(z)=\sum_{k=0}^{m-1} z^k C_k,
\qquad C_k\in\C^{s\times s}.
\end{equation*}
Its action on $\vec B$ through $A$ is defined by
\begin{equation}
  \label{eq:matrix-polynomial-action}
P(A)\circ\vec B
=
\sum_{k=0}^{m-1} A^k \vec B C_k.
\end{equation}
With this notation,
\begin{equation}
  \label{eq:block-krylov-subspace-definition-matrix-polynomials}
\bkryl_{m}(A,\vec B)
=
\{\, P(A)\circ\vec B : \deg P \leq m-1 \,\}.
\end{equation}
Thus, block Krylov subspaces are precisely the images of the action of matrix
polynomials evaluated at $A$ and applied to $\vec B$.

\subsection{Orthogonal matrix polynomials}

We briefly recall standard notions on orthogonal matrix polynomials (OMP) with
respect to matrix-valued measures; for comprehensive treatments we refer to,
e.g.,~\cite{DaPuSi08,SimonCMV07}.

Let $\poly^s$ denote the set of all $s\times s$ matrix polynomials, and let $\poly^s_d$ denote the subset of matrix polynomials in $\poly^s$ with degree at most $d$. Let $\bmu$ be a $\C^{s\times s}$-valued Hermitian positive semidefinite
measure supported on a contour $\Gamma\subset\C$.
Such a measure induces a \emph{right matrix-valued inner product} on matrix
polynomials via
\begin{equation*}
\llangle P,Q\rrangle_{\bmu} := \int_\Gamma P(z)^*\,\mathrm d\bmu(z)\,Q(z).
\end{equation*}
We assume that the integral is well-defined for all $P,Q\in\poly^s$.

We say that the inner product $\llangle\cdot,\cdot\rrangle_{\bmu}$ is nondegenerate if, for any monic matrix polynomial $P(z)$,\footnote{A matrix polynomial is monic if the leading coefficient is the identity matrix.} $\llangle P,P\rrangle_{\bmu}$ is positive definite. Equivalently, this means that 
\begin{equation*}
  \llangle P,P\rrangle_{\bmu} \neq 0\quad\text{for all $P\in\poly^s$, $P(z)\not\equiv 0$}.
\end{equation*}
Under this assumption, there exists
an infinite sequence $\{P_k\}_{k\ge 0}$ of \emph{orthonormal matrix polynomials}, such that $\deg P_k = k$ and
\begin{equation*}
\llangle P_i,P_j\rrangle_{\bmu} = \delta_{i,j}\,I_{s}.
\end{equation*}
If the inner product $\llangle\cdot,\cdot\rrangle_{\bmu}$ is nondegenerate only when restricted to $\poly^s_d$, i.e.,
\begin{equation*}
  \llangle P,P\rrangle_{\bmu} \neq 0\quad\text{for all $P\in\poly^s_d$, $P(z)\not\equiv 0$,}
\end{equation*}
then the existence of a finite sequence $P_0,\dots,P_d$ of orthonormal matrix polynomials is still guaranteed. We will still use the notation $\{ P_k \}_{k\geq 0}$ to denote such a sequence.

As in the scalar case, the multiplication by $z$ maps
$\poly^s_k$ into $\poly^s_{k+1}$, which leads to finite-term recurrence relations for $\{P_k\}_{k\geq 0}$:
\begin{equation}
  \label{eq:matrix-poly-recurrence-generic}
  zP_{k-1}(z) = \sum_{h=0}^{k} P_h(z)H_{h+1,k},\quad H_{h+1,k} \in \C^{s\times s}.
\end{equation}
When $\Gamma\subset\R$, these recurrences take the form of a three-term
relation represented by a block Jacobi matrix.
When $\Gamma\subset\T$, where $\T$ denotes the complex unit circle, one obtains a Szeg\H{o}-type recursion governed by
matrix-valued Verblunsky coefficients.
We will return in more detail to the unit circle case in \cref{sec:unitary-block-Krylov}, where
the theory will be related to block Krylov subspaces associated with unitary matrices.

\begin{remark}
  \label{rem:discrete-inner-product-degenerate}
 Let $\bmu$ be supported on a finite set $\{ z_1,\dots,z_N \}$, so
  \begin{equation}
    \label{eq:discrete-inner-product-matrix-polynomials}
    \llangle P,Q \rrangle_{\bmu} = \sum_{j=1}^N P(z_j)^* \bmu({z_j}) Q(z_j).
  \end{equation}
  In this case, the inner product is necessarily degenerate. In fact, $\llangle P,P\rrangle_{\bmu} = \vec 0$ whenever $P(z_j) = \vec 0$ for all $j=1,\dots,N$. However, $\llangle\cdot,\cdot\rrangle_{\bmu}$ can still be nondegenerate when restricted to $\poly^s_d$, for some $d\leq N$, and this guarantees the existence of a finite sequence of orthonormal matrix polynomials. This will be the case for inner products linked to nondeflating block Krylov subspaces.
\end{remark}

  Because of noncommutativity, one must distinguish between left and right orthogonality and the corresponding orthogonal polynomials. In fact, one can define the \emph{left matrix-valued inner product}
  \begin{equation*}
    \llangle P,Q\rrangle^L_{\bmu} = \int_{\Gamma}P(z)\,\mathrm d\bmu(z)\,Q(z)^*,
  \end{equation*}
  which leads to a different sequence of \emph{left orthonormal matrix polynomials} $\{ P^L_k \}_{k\geq 0}$. In~\cref{sec:block-krylov-orthogonal-matrix-polynomials-correspondence}, we will see that orthonormal bases of block Krylov subspaces are linked with right orthonormal polynomials. However, for measures supported on the unit circle, there is a special interaction between right and left polynomials that leads to the Szeg\H{o} recurrence. We will return to this in~\cref{sec:unitary-block-Krylov}.

\section{Block Krylov--orthogonal matrix polynomials correspondence}
\label{sec:block-krylov-orthogonal-matrix-polynomials-correspondence}

In this section we establish the connection between block Krylov subspaces and
orthogonal matrix polynomials through the matrix-valued spectral measure
induced by a pair $(A,\vec B)$.

\subsection{Isometry between block Krylov subspaces and matrix polynomials}

The pair $(A,\B)$, with $A\in\C^{n\times n}$ and $\vec B\in\C^{n\times s}$, defines an inner product on
matrix polynomials:
\begin{equation}
  \label{eq:matrix-poly-inner-product-generic}
\llangle P,Q\rrangle_{A,\vec B}
=
(P(A)\circ\vec B)^*\, Q(A)\circ\vec B.
\end{equation}

\begin{theorem}
  \label{thm:correspondence-basis-orthogonal-matrix-polynomials}
  Let $A\in\C^{n\times n}$ and $\B\in\C^{n\times s}$.
  \begin{enumerate}
    \item If $\V_1,\dots,\V_{m}$ is a nested orthonormal basis of $\bkryl_m(A,\B)$, then $\V_k = P_{k-1}(A)\circ\B$, $k=1,\dots,m$, where $P_0,\dots,P_{m-1}$ are orthonormal matrix polynomials with respect to $\llangle\cdot,\cdot\rrangle_{A,\B}$.
    \item If $P_0(z),\dots,P_{m-1}(z)$ are orthonormal matrix polynomials with respect to $\llangle\cdot,\cdot\rrangle_{A,\B}$, then the block vectors $\V_k=P_{k-1}(A)\circ\B$, $k=1,\dots,m$, form a nested orthonormal basis of $\bkryl_{m}(A,\B)$.
  \end{enumerate}
\end{theorem}
\begin{proof}
  For all $k=1,\dots,m$, $\V_k\in\bkryl_k(A,\B)$, so we can write $\V_k=P_{k-1}(A)\circ\B$ for some matrix polynomial $P_{k-1}\in\poly^s_{k-1}$. Since the basis is nested, $\deg(P_{k-1})=k-1$. Moreover, from the definition of $\llangle\cdot,\cdot\rrangle_{A,\B}$, we deduce that the polynomials are orthonormal since 
  \begin{equation}
    \label{eq:OMP-krylov-basis-equivalence}
  \llangle P_i,P_j\rrangle_{A,\B} = \V_{i+1}^*\V_{j+1}=\delta_{i,j}I_{s}.
  \end{equation}
  
  Conversely, if $P_0,\dots,P_{m-1}$ are orthonormal matrix polynomials, then $\V_k\in\bkryl_k(A,\B)$ for all $k$, and~\cref{eq:OMP-krylov-basis-equivalence} still holds. Thus, $\V_1,\dots,\V_{m}$ form a nested orthonormal basis of $\bkryl_{m}(A,\B)$.
\end{proof}

As a consequence of~\cref{thm:correspondence-basis-orthogonal-matrix-polynomials}, for nondeflating block Krylov subspaces, there is a one-to-one correspondence between nested orthonormal bases of $\bkryl_m(A,\B)$ and orthonormal matrix polynomials with respect to $\llangle\cdot,\cdot\rrangle_{A,\B}$.

The following theorem clarifies the link between matrix polynomials and block Krylov subspaces in the nondeflating case.
\begin{theorem}
  \label{thm:equivalence-nodeflation-nondegenerate-isomorphism}
  Let $\bkryl_m(A,\B)$ be the block Krylov subspace equipped with the matrix-valued Euclidean inner product, and let $\poly^s_{m-1}$ be the space of $s\times s$ matrix polynomials with degree at most $m-1$ equipped with $\llangle\cdot,\cdot\rrangle_{A,\B}$. Consider the linear mapping 
  \begin{equation*}
\mathcal T :
\poly_{m-1}^{s}
\longrightarrow
\bkryl_{m}(A,\vec B),
\qquad
\mathcal T(P)=P(A)\circ\vec B.
\end{equation*}
The following are equivalent.
\begin{enumerate}
  \item No deflation occurs in $\bkryl_m(A,\B)$.
  \item The mapping $\mathcal{T}$ is an isometric isomorphism.
  \item The inner product $\llangle\cdot,\cdot\rrangle_{A,\B}$ is nondegenerate over $\poly^s_{m-1}$.
\end{enumerate}

\end{theorem}
\begin{proof}
  We prove that $(1) \implies (2) \implies (3) \implies (1)$.

  \emph{$(1)\implies(2)$.}
  The mapping $\mathcal{T}$ is surjective, because of~\cref{eq:block-krylov-subspace-definition-matrix-polynomials}. By~\cref{prop:no-deflation-block-krylov}, the dimension of $\bkryl_m(A,\B)$ as a $\C$-vector space is $s^2m$, which coincides with the dimension of $\poly^s_{m-1}$. Hence, $\mathcal{T}$ is a surjective mapping between spaces with the same dimension, and thus an isomorphism. It is also isometric since, by definition,
  \begin{equation*}
    \llangle P,Q\rrangle_{A,\B} = \mathcal{T}(P)^*\,\mathcal{T}(Q).
  \end{equation*}

  \emph{$(2)\implies(3)$.} Suppose that $\llangle\cdot,\cdot\rrangle_{A,\B}$ is degenerate, so there exists $P(z)\in\poly^s_{m-1}$, $P(z)\not\equiv 0$, such that
  \begin{equation*}
    \llangle P,P\rrangle_{A,\B} = (P(A)\circ\B)^*\,(P(A)\circ\B) = \vec 0.
  \end{equation*}
  Then, $P(A)\circ \B=\vec 0$, since it has the same rank as $\llangle P,P\rrangle_{A,\B}$. Hence $\mathcal{T}$ is not injective, contradicting the fact that $\mathcal{T}$ is an isomorphism.

  (3)$\implies$(1). If $\llangle\cdot,\cdot\rrangle_{A,\B}$ is nondegenerate over $\poly^s_{m-1}$, there exists a sequence $P_0,\dots,P_{m-1}$ of orthonormal matrix polynomials. Then, $\bkryl_m(A,\B)$ admits an orthonormal basis of block vectors, made by 
  \begin{equation*}
    \V_k = P_{k-1}(A)\circ \B,\quad k=1,\dots,m.
  \end{equation*}
  Hence, $\bkryl_m(A,\B)$ does not deflate by~\cref{prop:no-deflation-block-krylov}.
\end{proof}

\subsection{Spectral measures induced by $(A,\vec B)$}
For normal matrices, the inner product $\llangle\cdot,\cdot\rrangle_{A,\B}$ admits a spectral representation.

\begin{theorem}
  \label{thm:matrix-valued-inner-product-poly-discrete}
  Let $A$ be normal with the orthonormal eigendecomposition
  \begin{equation}
    \label{eq:eigendecomposition}
    A = \sum_{j=1}^n \lambda_j \vec u_j \vec u_j^*.
  \end{equation}
  Then, for all matrix polynomials $P(z)$ and $Q(z)$,
  \begin{equation}
    \label{eq:matrix-poly-discrete-inner-product}
    \llangle P,Q\rrangle_{A,\B}
    =
    \sum_{j=1}^n
    P(\lambda_j)^*
    \big( \vec u_j^* \vec B \big)^*
    \big( \vec u_j^* \vec B \big)
    Q(\lambda_j).
  \end{equation}
\end{theorem}
\begin{proof}
  Let
  \begin{equation*}
    P(z) = \sum_{k=0}^d z^kC_k.
  \end{equation*}
  Expanding $P(A)\circ \B$ with~\cref{eq:matrix-polynomial-action}, together with the eigendecomposition~\cref{eq:eigendecomposition}, gives us the following identity:
  \begin{align}
    \nonumber
    P(A)\circ\B = \sum_{k=0}^d A^k\B C_k
    &= \sum_{k=0}^d \left( \sum_{j=1}^n \lambda_j^k\vec u_j\vec u_j^* \right)\B C_k\\
    \nonumber
    &= \sum_{j=1}^n \vec u_j\vec u_j^* \B\left( \sum_{k=0}^d  \lambda_j^kC_k \right)\\ \label{eq:spectral-identity-action-P}
    &= \sum_{j=1}^n \vec u_j\vec u_j^* \B P(\lambda_j).
  \end{align}
  Similarly, for $Q(A)\circ\B$ we get
  \begin{equation}
    \label{eq:spectral-identity-action-Q}
    Q(A)\circ\B = \sum_{j=1}^n \vec u_j\vec u_j^* \B Q(\lambda_j).
  \end{equation}
  By inserting~\cref{eq:spectral-identity-action-P} and~\cref{eq:spectral-identity-action-Q} into the formula $(P(A)\circ \B)^*(Q(A)\circ \B)$, we get
  \begin{align*}
    \llangle P,Q\rrangle_{A,\B} &= \left( \sum_{i=1}^n P(\lambda_i)^* \B^* \vec u_i \vec u_i^* \right) \left( \sum_{j=1}^n \vec u_j\vec u_j^* \B Q(\lambda_j)\right)\\
    &= \sum_{i=1}^n \sum_{j=1}^n P(\lambda_i)^* \B^* \vec u_i (\vec u_i^* \vec u_j)\vec u_j^* \B Q(\lambda_j)\\
    &= \sum_{j=1}^n P(\lambda_j)^* \B^* \vec u_j \vec u_j^* \B Q(\lambda_j),
  \end{align*}
  where we have used the fact that $\vec u_1,\dots,\vec u_n$ are orthonormal vectors.
\end{proof}

\begin{remark}
  A similar representation for the matrix-valued inner product $\llangle P,Q\rrangle_{A,\B}$ has been derived in~\cite{LundPhD} for the Hermitian case. Here, our formulation works for any normal matrix $A$, since it relies only on a spectral decomposition.
\end{remark}

The identity~\cref{eq:matrix-poly-discrete-inner-product} suggests introducing a matrix-valued measure $\bmu$ supported on the spectrum of $A$ by defining, for $\Omega\subset\C$,
\begin{equation}
  \label{eq:spectral-measure}
  \bmu(\Omega)
  =
  \sum_{\lambda_j\in\Omega} 
  (\vec u_j^*\vec B)^*(\vec u_j^*\vec B).
\end{equation}
The measure $\bmu$ is Hermitian positive semidefinite, and its support consists of those eigenvalues of $A$ whose eigenspaces have nonzero components in the columns of $\vec B$. Moreover, the rank of $\bmu(\{\lambda_j\})$ equals the dimension of the projection of $\vec B$ onto the eigenspace corresponding to $\lambda_j$. In particular, if an eigenspace is orthogonal to $\vec B$, then it does not contribute to the inner product and is orthogonal to the block Krylov subspace $\kryl_{m}^{\square}(A,\B)$.

With this notation, the inner product induced by $(A,\vec B)$ can be written as
\begin{equation}
  \label{eq:poly-inner-product-integral-representation}
\llangle P,Q\rrangle_{A,\vec B}
=
\int_{\sigma(A)}
P(z)^* \, \mathrm d\bmu(z) \, Q(z).
\end{equation}
This representation has important consequences. When $A$ is Hermitian, the measure $\bmu$ is supported on the real line, and the three-term recurrence of the orthogonal matrix polynomials coincides with the recurrence given by the block Lanczos algorithm. When $A$ is unitary, any of its eigenvalues has unit modulus, so $\bmu$ is supported on the unit circle, and thus the associated orthogonal matrix polynomials satisfy the Szeg\H{o} recurrence relations~\cite{DaPuSi08,SinVanAs94}. We exploit this property in~\cref{sec:unitary-block-Krylov}.

\subsection{Deflated block Krylov subspaces and singular matrix polynomials}
To describe the relation in case of deflation, we need some notions on rectangular matrix polynomials and their action on block vectors.

A rectangular matrix polynomial has the form
\begin{equation*}
  \widetilde{P}(z) = \sum_{k=0}^d z^k \ttilde{C}_k,\qquad \ttilde{C}_k\in\C^{s\times r}.
\end{equation*}
For rectangular matrix polynomials, we can still define their action on $A\in\C^{n\times n}$ and $\B\in\C^{n\times s}$ as
\begin{equation*}
  \ttilde{P}(A)\circ\B = \sum_{k=0}^d A^k\B\ttilde{C}_k.
\end{equation*}
Note that $\ttilde{P}(A)\circ\B\in\C^{n\times r}$.

Consider~\cref{alg:block-arnoldi}. At the $k$-th step, we have the block vectors $\V_1,\dots,\V_k\in\C^{n\times s}$ which form an orthonormal basis of $\bkryl_k(A,\B)$. 
Suppose that a deflation is encountered, so
\begin{equation*}
  \W_{k+1} = A\V_k-\sum_{h=1}^k V_h H_{h,k},\quad H_{h,k} = \V_h^* A \V_k,
\end{equation*}
has rank $r<s$. Hence, it is impossible to compute an orthonormal basis of $\bkryl_{k+1}(A,\B)$. However, we can compute a rank-revealing factorization
\begin{equation*}
  \W_{k+1} = \ttilde{\V}_{k+1} \ttilde{R},\qquad \ttilde{\V}_{k+1}\in\C^{n\times r}, \quad \ttilde{R} \in\C^{r\times s},
\end{equation*}
such that 
\begin{equation*}
  \ttilde{\V}_{k+1}^*\ttilde{\V}_{k+1} = I_r,\qquad \V_i^*\ttilde{\V}_{k+1} = \vec 0_{s\times r}\quad\text{for all $i\leq k$}.
\end{equation*} 
Since $\ttilde{R}$ has rank $r$, consider any matrix $\ttilde{S}\in\C^{s\times r}$ such that\footnote{For example, the Moore-Penrose pseudoinverse of $\ttilde{R}$.}
\begin{equation*}
  \ttilde{R}\,\ttilde{S} = I_r.
\end{equation*}
Then,
\begin{equation*}
  \ttilde{\V}_{k+1} = \W_{k+1}\ttilde{S} = A\V_k\ttilde{S} - \sum_{h=1}^k V_h H_{h,k} \ttilde{S}.
\end{equation*}

We can link $\ttilde{\V}_{k+1}$ to a rectangular matrix polynomial as follows. Let $P_0,\dots,P_k$ be the $s\times s$ orthogonal matrix polynomials such that
\begin{equation*}
  \V_j = P_{j-1}(A)\circ\B,\qquad \llangle P_i,P_j\rrangle_{A,\B} = \delta_{i,j}I_{s}.
\end{equation*}
Then, define
\begin{equation*}
  \ttilde{P}_{k+1}(z) = zP_k(z)\ttilde{S} - \sum_{h=1}^{k} P_{h-1}(z)H_{h,k}\ttilde{S}.
\end{equation*}
This $s\times r$ matrix polynomial satisfies
\begin{equation*}
  \ttilde{P}_{k+1}(A)\circ\B = \ttilde{\V}_{k+1}.
\end{equation*}
Moreover, if $A$ is normal and $\llangle\cdot,\cdot\rrangle_{A,\B}$ has the expression~\cref{eq:poly-inner-product-integral-representation}, we can write
\begin{equation*}
  \int_{\sigma(A)} \ttilde{P}_{k+1}(z)^* \, \de \bmu(z) \, \ttilde{P}_{k+1}(z) = I_r,\qquad 
  \int_{\sigma(A)} P_j(z)^* \, \de \bmu(z) \, \ttilde{P}_{k+1}(z) = \vec 0_{s\times r}\quad\text{for all $j\leq k$.}
\end{equation*}
The correspondence we have just established may motivate further analysis of degenerate matrix-valued inner products and the synergy between square and rectangular orthogonal matrix polynomials.

\subsection{Matrix-valued rational functions and rational block Krylov subspaces}
\label{subsec:rational-block-krylov}

The connection between orthogonal matrix polynomials and block Krylov subspaces extends to the rational case. Consider $A\in\C^{n\times n}$, $\B\in\C^{n\times s}$, and a set of poles $\btheta_{m-1} = (\theta_1,\dots,\theta_{m-1})$, such that each finite pole lies outside the spectrum of $A$. The \emph{block rational Krylov subspace} associated with $A,\B,\btheta_{m-1}$ is
\begin{equation}
  \label{eq:rational-block-krylov-subspace}
  \rbkryl_{m}(A,\B,\btheta_{m-1}) = q_{m-1}(A)^{-1}\bkryl_m(A,\B), 
\end{equation}
where
\begin{equation*}
  q_{m-1}(z) = \prod_{\substack{k=1\\\theta_k\neq\infty}}^{m-1} (z-\theta_k).
\end{equation*}
Note that the degree of $q_{m-1}$ equals the number of finite poles. In particular, if each pole is infinite, $q_{m-1}\equiv 1$ and the $\rbkryl_m(A,\B,\btheta_{m-1}) = \bkryl_m(A,\B)$.

 Elsworth and G\"uttel~\cite{ElsGutt20} used matrix-valued rational functions to characterize~\cref{eq:rational-block-krylov-subspace} in the following way: given
 \begin{equation}
  \label{eq:matrix-valued-rational-function-generic}
  R(z) = q_{m-1}(z)^{-1}P(z),\quad P(z) \in \poly^s,
 \end{equation}
the action of $R$ on the pair $(A,\B)$ is
\begin{equation*}
  R(A)\circ \B = q_{m-1}(A)^{-1} P(A)\circ\B.
\end{equation*}
Let $\mathcal{R}^{s}_{m-1}(q_{m-1}) := q_{m-1}^{-1}\poly^{s}_{m-1}$ denote the set of rational functions of the form~\cref{eq:matrix-valued-rational-function-generic} with $P(z)\in\poly^{s}_{m-1}$. Then
\begin{equation}
  \label{eq:rational-block-krylov-functional-representation}
  \rbkryl_{m}(A,\B,\btheta_{m-1}) = \left\{ R(A)\circ\B : R(z)\in \mathcal{R}^{s}_{m-1}(q_{m-1}) \right\}.
\end{equation}
The polynomial inner product in~\cref{eq:matrix-poly-inner-product-generic} extends to this setting: for $R,S\in \mathcal{R}^{s}_{m-1}(q_{m-1})$, 
\begin{equation*}
  \llangle R,S\rrangle_{A,\B} = (R(A)\circ \B)^*\,(S(A)\circ\vec B).
\end{equation*}
If $A$ is normal, we get the analogue of~\cref{thm:matrix-valued-inner-product-poly-discrete}.
\begin{theorem}
  \label{thm:matrix-valued-inner-product-rational-discrete}
  Let $A$ be normal with the orthonormal eigendecomposition~\cref{eq:eigendecomposition}. Then, for all matrix-valued rational functions $R(z)$ and $S(z)$ of the form~\cref{eq:matrix-valued-rational-function-generic}, we have
  \begin{equation*}
    \llangle R,S\rrangle_{A,\B} = \sum_{j=1}^n
    R(\lambda_j)^*
    \big( \vec u_j^* \vec B \big)^*
    \big( \vec u_j^* \vec B \big)
    S(\lambda_j).
  \end{equation*}
\end{theorem}
\begin{proof}
  Let $R(z)=q_{m-1}(z)^{-1} P(z)$, where $P(z)$ is a matrix polynomial. From~\cref{eq:spectral-identity-action-P}, we get
  \begin{align*}
    R(A)\circ\B = q_{m-1}(A)^{-1}\,(P(A)\circ\B)
    &=
    \left(\sum_{j=1}^n q_{m-1}(\lambda_j)^{-1} \vec u_j\vec u_j^* \right)
    \left(\sum_{j=1}^n \vec u_j\vec u_j^* \B P(\lambda_j)\right)\\
    &=
    \sum_{j=1}^n \vec u_j\vec u_j^* \B \,q_{m-1}(\lambda_j)^{-1}P(\lambda_j)\\
    &=
    \sum_{j=1}^n \vec u_j\vec u_j^* \B R(\lambda_j).
  \end{align*}
  Similarly,
  \begin{equation*}
    S(A)\circ\B = \sum_{j=1}^n \vec u_j\vec u_j^* \B S(\lambda_j).
  \end{equation*}
  With these expansions for $R(A)\circ\B$ and $S(A)\circ\B$, the remainder of the proof is identical to that of~\cref{thm:matrix-valued-inner-product-poly-discrete}.
\end{proof}
A consequence of~\cref{thm:matrix-valued-inner-product-rational-discrete} is that
\begin{equation*}
  \llangle R,S\rrangle_{A,\B} = \int_{\sigma(A)} R(z)^*\, \mathrm d\bmu(z) \,S(z) = \llangle R,S\rrangle_{\bmu},
\end{equation*}
where $\bmu$ is the spectral measure~\cref{eq:spectral-measure} associated with $(A,\B)$. Hence, an orthonormal basis $\V_1,\dots,\V_{m}$ of $\rbkryl_{m}(A,\B,\btheta_{m-1})$ corresponds to a sequence $R_0(z),\dots,R_{m-1}(z)$ of orthonormal matrix-valued rational functions with respect to the inner product $\llangle\cdot,\cdot\rrangle_{\bmu}$.

\subsection{Laurent matrix polynomials and extended block Krylov}
\label{subsec:laurent-extended-block-krylov}

Let $A\in\C^{n\times n}$ and $\B\in\C^{n\times s}$, and consider the sequence of block vectors
\begin{equation}
  \label{eq:generators-extended-block-krylov}
  \B, A\B, A^{-1}\B, A^2\B, A^{-2}\B, \dots .
\end{equation}
The \emph{extended block Krylov subspace} $\ebkryl_m(A,\B)$ is the subspace of $\C^{n\times s}$ generated by the first $m$ block vectors in~\cref{eq:generators-extended-block-krylov}. More precisely, 
\begin{equation}
  \label{eq:extended-block-krylov-blockspan}
  \begin{aligned}
    \ebkryl_{2\ell+1}(A,\B) &= \blockspan\{ \B,A\B,A^{-1}\B,\dots,A^{\ell}\B,A^{-\ell}\B \},\\
    \ebkryl_{2\ell+2}(A,\B) &= \blockspan\{ \B,A\B,A^{-1}\B,\dots,A^{\ell}\B,A^{-\ell}\B,A^{\ell+1}\B \}.
  \end{aligned}
\end{equation}
Note that $\ebkryl_m(A,\B)$ is a particular instance of a rational block Krylov subspace, where the set of poles is $\btheta_{m-1}=(\infty,0,\infty,0,\dots)$. Furthermore,
\begin{equation}
  \label{eq:extended-polynomial-block-krylov-correspondence}
  \ebkryl_m(A,\B) = A^{-\ell} \bkryl_m(A,\B),\quad \ell = \left\lfloor \frac{m-1}{2}\right\rfloor.
\end{equation}

The matrix-valued rational functions linked to $\ebkryl_m(A,\B)$ are Laurent matrix polynomials,
\begin{equation*}
  R(z) = \sum_{k=-d_1}^{d_2} z^k C_k,
  \quad
  C_k\in\C^{s\times s}.
\end{equation*}
Define $\laur^s_m$ as the subset of Laurent matrix polynomials generated by the first $m$ terms of the sequence\footnote{Note that this notation is quite different from the one used for polynomials: in $\poly^s_d$, $d$ is the maximal degree, not the number of generators.}
\begin{equation*}
  I_s, zI_s, z^{-1}I_s, z^2I_s, z^{-2}I_s, \dots .
\end{equation*}
Alternatively, 
\begin{equation}
  \label{eq:definition-laurent-odd-even}
  \begin{aligned}
    \laur^s_{2\ell+1} &= \left\{ \sum_{k=-\ell}^\ell z^k C_k: C_k\in\C^{s\times s}\right\} = z^{-\ell}\poly^s_{2\ell},\\
     \laur^s_{2\ell+2} &= \left\{ \sum_{k=-\ell}^{\ell+1} z^k C_k: C_k\in\C^{s\times s}\right\} = z^{-\ell}\poly^s_{2\ell+1}.
  \end{aligned}
\end{equation}

From~\cref{eq:extended-block-krylov-blockspan} and~\cref{eq:definition-laurent-odd-even}, it follows that
\begin{equation}
  \label{eq:extended-block-krylov-representation-laurent-matrix-polynomials}
  \ebkryl_m(A,\B) = \{ R(A)\circ\B : R(z)\in\laur^s_{m} \}.
\end{equation}

Thanks to the correspondence between polynomial and extended block Krylov subspaces, we have the following result.
\begin{theorem}
  \label{thm:isomorphism-extended-laurent}
  Let $A$ be nonsingular. The dimension of $\ebkryl_m(A,\B)$ as a $\C$-vector space is $s^2m$ if and only if the inner product $\llangle\cdot,\cdot\rrangle_{A,\B}$ is nondegenerate over $\poly^s_{m-1}$. In this case, the mapping 
  \begin{equation*}
    \mathcal{T}:\laur^s_{m}\to \ebkryl_m(A,\B),\quad \mathcal{T}(R) = R(A)\circ\B,
  \end{equation*}
  is an isometric linear isomorphism, according to $\llangle\cdot,\cdot\rrangle_{A,\B}$ and the matrix-valued Euclidean inner product.
\end{theorem} 
\begin{proof}
  From~\cref{eq:extended-polynomial-block-krylov-correspondence}, it follows that the dimensions of $\ebkryl_{m}(A,\B)$ and $\bkryl_{m}(A,\B)$ as $\C$-vector spaces coincide. Thus, the first claim follows from~\cref{thm:equivalence-nodeflation-nondegenerate-isomorphism}. 

  For the second claim, we observe that $\mathcal{T}$ is surjective by~\cref{eq:extended-block-krylov-representation-laurent-matrix-polynomials}. Moreover, the $\C$-dimension of $\laur^s_m$ is $s^2m$. Then, $\mathcal{T}$ is an isomorphism if and only if the $\C$-dimension of $\ebkryl_m(A,\B)$ is $s^2m$, and the isometry follows from the definition of $\llangle\cdot,\cdot\rrangle_{A,\B}$.
\end{proof}

As a consequence of \cref{thm:isomorphism-extended-laurent}, orthonormal bases of $\ebkryl_m(A,\B)$ are in correspondence with orthonormal functions in $\laur^s_m$ with respect to $\llangle\cdot,\cdot\rrangle_{A,\B}$. We use this link in~\cref{sec:unitary-extended-block-krylov-CMV} to relate CMV matrices with extended block Krylov subspaces associated with unitary matrices.

\section{Unitary case: Szeg\H{o} recurrence and block isometric Arnoldi}
\label{sec:unitary-block-Krylov}

In this section we report the relevant results on orthogonal matrix polynomials on the unit circle (OMPUC), and apply them to the orthogonalization of block Krylov subspaces associated with unitary matrices.

\subsection{OMPUC and Szeg\H{o} recurrence}
\label{subsec:szego}

Let $\bmu$ be a matrix-valued measure supported on $\T$. We consider the right inner product
\begin{equation*}
\llangle P,Q\rrangle_{\bmu}
=
\int_{\T} P(z)^*\,\de\bmu(z)\,Q(z).
\end{equation*}
Assume that $\llangle\cdot,\cdot\rrangle_{\bmu}$ is nondegenerate, so that there exists a sequence $\{\varphi_k^R\}_{k\ge0}$, with $\deg \varphi_k^R=k$, of right orthonormal matrix polynomials that satisfy
\begin{equation*}
\llangle \varphi_i^R,\varphi_j^R\rrangle_{\bmu}
=
\delta_{i,j}I_{s}.
\end{equation*}
Similarly, we consider a sequence $\{\varphi_k^L\}_{k\ge0}$ of left
orthonormal matrix polynomials, that satisfy
\begin{equation*}
\llangle \varphi^L_i,\varphi^L_j\rrangle_{\bmu}^L 
=
\int_{\T} \varphi_i^L(z)\,\de\bmu(z)\,\varphi_j^L(z)^*
=
\delta_{i,j}I_{s}.
\end{equation*}

For $P(z)=\sum_{j=0}^k z^jP_j\in\poly^s$ we define the reversed polynomial
\begin{equation*}
P^{\#}(z)
=
z^k P(1/\overline z)^*
=
\sum_{j=0}^k z^{k-j}P_j^*.
\end{equation*}

There exists a sequence
of matrices $\{\malpha_k\}_{k\ge1}$, with
\begin{equation*}
\malpha_k\in\C^{s\times s},
\qquad
\|\malpha_k\|_2<1,
\end{equation*}
called the \emph{matrix Verblunsky coefficients}, such that
\begin{align}
  \label{eq:OMPUC-szego-recurrence}
  \begin{aligned}
    z\varphi^{R}_{k}(z) &= \varphi^R_{k+1}(z)\mrho^R_{k+1} + \varphi^{L,\#}_k \malpha_{k+1}^*,\\
    \varphi^{L,\#}_k &= \varphi^{L,\#}_{k+1} \mrho^L_{k+1} + z\varphi^R_k \malpha_{k+1},
  \end{aligned}
\end{align}
where $\varphi^{L,\#}_k=(\varphi^L_k)^{\#}$, and
\begin{equation}
  \label{eq:rho-formula}
  \mrho^R_{k+1}=(I-\malpha_{k+1}\malpha_{k+1}^*)^{1/2},
  \quad
  \mrho^L_{k+1}=(I-\malpha_{k+1}^*\malpha_{k+1})^{1/2}.
\end{equation}
The relation~\cref{eq:OMPUC-szego-recurrence} is the \emph{Szeg\H{o} recurrence}, the unit-circle analogue of the three-term recurrence on the real line. 

By Verblunsky's Theorem~\cite[Theorem 3.12]{DaPuSi08}, the matrices $\{\malpha_k\}_{k\ge1}$ provide a one-to-one correspondence with the 
matrix-valued measure $\bmu$ on $\T$ and, in turn, with the associated OMPUC. Furthermore, by~\cref{eq:OMPUC-szego-recurrence}, the Verblunsky coefficients play the role of recurrence parameters in the Szeg\H{o} theory. Note that, from the definition of $\varphi^{L,\#}_k$, it follows that
\begin{equation}
  \label{eq:OMPUC-left-reversed-normality}
  \llangle\varphi^{L,\#}_k,\varphi^{L,\#}_k\rrangle_{\bmu} = \int_{\T} \varphi^L_k(z) \,\de\bmu(z)\,\varphi^L_k(z)^* = I_{s}.
\end{equation}
If we consider the first relation in~\cref{eq:OMPUC-szego-recurrence}, we get the following identity:
\begin{equation*}
  \llangle \varphi^{L,\#}_k, z\varphi^R_k\rrangle_{\bmu} = \llangle\varphi^{L,\#}_k,\varphi^R_{k+1}\rrangle_{\bmu}\, \mrho^R_{k+1} + \llangle\varphi^{L,\#}_k,\varphi^{L,\#}_k\rrangle_{\bmu}\, \malpha_{k+1}^* = \malpha_{k+1}^*,
\end{equation*}
where we have used~\cref{eq:OMPUC-left-reversed-normality} and the fact that $\varphi^R_{k+1}$ is orthogonal to any matrix polynomial in $\poly_k^{s}$. Thus,
\begin{equation}
  \label{eq:verblunsky-formula-OMPUC}
  \malpha_{k+1} = \llangle \varphi^{L,\#}_k,z\varphi^R_k\rrangle_{\bmu}^* = \llangle z\varphi^R_k, \varphi^{L,\#}_k\rrangle_{\bmu}.
\end{equation}

\begin{remark}
  Note that the recurrence~\cref{eq:OMPUC-szego-recurrence} is also valid when $\llangle\cdot,\cdot\rrangle_{\bmu}$ is nondegenerate over $\poly^s_{d}$ for some positive integer $d$, so that only the existence of a finite sequence of left and right OMPUC is guaranteed. See the proof of~\cite[Theorem 2.13.4]{Simon-book05}.
\end{remark}

\begin{remark}
  In the monic setting, the Verblunsky coefficients can be obtained from point evaluations of the orthogonal polynomials. More precisely, if $\{\Phi_k\}_{k\geq 0}$ denotes the sequence of monic right orthogonal polynomials, then
  \begin{equation*}
    \malpha_{k} = \Phi_{k}(0)^*.
  \end{equation*}
  While this characterization is theoretically transparent, it is not well suited to a Krylov framework, since polynomials are accessed only through their action on $(A,\B)$. In contrast, the formula in~\cref{eq:verblunsky-formula-OMPUC} is more suited for our setting.
\end{remark}

\begin{remark}
  The Szeg\H{o} relation~\cref{eq:OMPUC-szego-recurrence} involves the right OMPUC $\{\varphi^R_k\}_{k\geq 0}$ and the reversed left OMPUC $\{\varphi^{L,\#}_k\}_{k\geq 0}$. An analogous relation holds between the sequences $\{ \varphi^L_k \}_{k\geq0}$ and $\{ (\varphi^{R}_k)^{\#} \}_{k\geq0}$; cf. \cite{DaPuSi08}. Since it is not relevant for the link with block Krylov subspaces, we omit the discussion here.
\end{remark}

\subsection{Unitary block Krylov and isometric Arnoldi}

Assume that $A$ is unitary, i.e.,
\begin{equation*}
A^*A = AA^* = I.
\end{equation*}
Since the eigenvalues of $A$ have unit modulus, the spectral measure $\bmu$ induced by $(A,\B)$ is supported on
the unit circle $\T$. Hence, the inner product induced by $(A,\B)$ has the form
\begin{equation*}
  \llangle P,Q\rrangle_{A,\B} = \int_{\mathcal{T}} P(z)^*\,\de\bmu(z)\,Q(z).
\end{equation*}
Through the spectral measure $\bmu$, we can also define the corresponding left inner product:
\begin{equation*}
  \llangle P,Q\rrangle_{A,\B}^L = \int_{\mathcal{T}} P(z)\,\de\bmu(z)\,Q(z)^*.
\end{equation*}

If the block Krylov subspace $\bkryl_m(A,\B)$ is nondeflating, then $\llangle\cdot,\cdot\rrangle_{A,\B}$ is nondegenerate over $\poly^s_{m-1}$ by~\cref{thm:equivalence-nodeflation-nondegenerate-isomorphism}. Under this assumption, there is a sequence $\varphi^R_0,\dots,\varphi^R_{m-1}$ of right OMPUC, and a sequence $\varphi^L_0,\dots,\varphi^L_{m-1}$ of left OMPUC. Through these sequences, we can construct the following block vectors:
\begin{equation*}
\V_{k}=\varphi_{k-1}^R(A)\circ\B,
\qquad
\widetilde \V_{k}=\varphi_{k-1}^{L,\#}(A)\circ\B,\qquad k=1,\dots,m.
\end{equation*}
From~\cref{thm:equivalence-nodeflation-nondegenerate-isomorphism}, the block vectors $\V_1,\dots,\V_{m}$ form a nested orthonormal basis of $\bkryl_m(A,\B)$, i.e.,
\begin{equation*}
  \V_i^*\V_j = \llangle \varphi^R_{i-1},\varphi^R_{j-1}\rrangle_{\bmu} = \delta_{i,j}I_{s},\quad \bkryl_k(A,\B) = \blockspan\{ \V_1,\dots,\V_k \},
\end{equation*}
for all $i,j$ and $k\leq m$. Furthermore, the Szeg\H{o} recurrence~\cref{eq:OMPUC-szego-recurrence} valid for matrix polynomials induces the following relation between the $\V_k$ and $\widetilde{\V}_k$:
\begin{equation}
  \label{eq:block-isometric-arnoldi}
  \begin{aligned}
    A\V_k &= \V_{k+1}\mrho^R_k + \widetilde{\V}_k \malpha_k^*,\\
    \widetilde{\V}_k &= \widetilde{\V}_{k+1}\mrho^L_k + A\V_k\malpha_k,
  \end{aligned}
\end{equation}
for all $k\geq 1$.
By~\cref{eq:verblunsky-formula-OMPUC}, we have
\begin{equation}
  \label{eq:alpha-formula-block-krylov}
  \malpha_k = \V_k^* A^* \widetilde{\V}_k,
\end{equation}
and $\mrho^R_k, \mrho^L_k$ are defined according to~\cref{eq:rho-formula}. 
The relation~\cref{eq:block-isometric-arnoldi} can be used to iteratively construct $\V_{k},\widetilde{\V}_{k}$. In fact,
\begin{equation*}
  \begin{aligned}
    \V_{k+1} &= \W_{k+1} (\mrho^R_k)^{-1},\quad \W_{k+1} = A\V_k - \widetilde{\V}_k\malpha_k^*,\\
    \widetilde{\V}_{k+1} &= \widetilde{\W}_{k+1}(\mrho^L_k)^{-1},\quad \widetilde{\W}_{k+1} = \widetilde{\V}_k - A\V_k\malpha_k.
  \end{aligned}
\end{equation*}
Note that $\|\malpha_k\|_2<1$, so the matrices $\mrho^R_k$ and $\mrho^L_k$ are both Hermitian positive definite and, in particular, invertible.
We can implement this strategy to generate an orthonormal basis of $\bkryl_m(A,\B)$. The induced procedure is reported in~\cref{alg:block-isometric-arnoldi}, and is a generalization of the \emph{isometric Arnoldi process}, introduced by Gragg in~\cite{Gragg93isometric}.

\begin{algorithm}[ht]
\caption{Block isometric Arnoldi}
\label{alg:block-isometric-arnoldi}
\begin{algorithmic}[1]
\REQUIRE Unitary $A\in\C^{n\times n}$, starting block vector $\B\in\C^{n\times s}$ such that $\B^*\B=I_{s}$, number of steps $m$
\ENSURE Orthonormal basis $\vec V_1,\dots,\vec V_{m}$ of $\bkryl_{m}(A,\B)$, auxiliary block vectors $\widetilde{\V}_1,\dots,\widetilde{\V}_{m}$, and Verblunsky coefficients $\malpha_1,\dots,\malpha_{m-1}$

\STATE \textbf{Initialize.} Set $\V_1 = \widetilde{\V}_1 = \B$

\FOR{$k=1,2,\dots,m-1$}

    \STATE \textbf{Multiplication with $A$.} Compute $\vec X_{k+1} = A\V_k$\label{line:matrix-block-vector-product}
    \STATE \textbf{Verblunsky coefficient.} Compute $\malpha_k = \vec X_{k+1}^* \widetilde{\V}_k$
    \STATE Compute $\Y_{k+1} = \X_{k+1} - \widetilde{\V}_{k}\malpha^*_k$ and $\widetilde{\Y}_{k+1} = \widetilde{\V}_k - \X_{k+1}\malpha_k$\label{line:verblunsky1}
    \STATE Form $\mrho^R_k = (I_{s}-\malpha_k\malpha_k^*)^{\frac{1}{2}}$, $\mrho^L_k = (I_{s}-\malpha_k^*\malpha_k)^{\frac{1}{2}}$
    \STATE Compute $\V_{k+1} = \Y_{k+1}(\mrho^R_k)^{-1}$ and $\widetilde{\V}_{k+1} = \widetilde{\Y}_{k+1}(\mrho^L_k)^{-1}$

\ENDFOR
\end{algorithmic}
\end{algorithm}

We now analyze the cost-per-iteration of~\cref{alg:block-isometric-arnoldi}. The cost of computing $\X_{k+1}$ at line~\ref{line:matrix-block-vector-product} depends on the structure of $A$; for instance, if $A$ is sparse, it is $O(s\operatorname{nnz}(A))$, where $\operatorname{nnz}(A)$ is the number of nonzero entries. The cost of computing $\malpha_k$, $\Y_{k+1}, \widetilde{\Y}_{k+1}$, $\V_{k+1}$, $\widetilde{\V}_{k+1}$ is $O(s^2n)$. Computing  $\mrho^R_k$, $\mrho^L_k$ via~\cref{eq:rho-formula}, i.e., forming $I_s-\malpha_k\malpha_k^*$, $I_s-\malpha_k^*\malpha_k$ and extracting the square roots, is $O(s^3)$, and thus negligible compared to the other operations provided that $s\ll n$.

\subsection{Schur parametrization of the block Hessenberg matrix of recurrences}

The Szeg\H{o} recurrence provides an efficient short recurrence for orthogonalizing the block Krylov subspace, involving the orthonormal basis $\V_1,\dots,\V_{m}$ and the auxiliary block vectors $\widetilde{\V}_1,\dots,\widetilde{\V}_{m}$. We now show that the same recurrence also yields an explicit parametrization of the block Hessenberg matrix of recurrences, extending the results valid for block-size $s=1$; see, e.g.,~\cite{HelKuiVanba05}.

Let
\begin{equation*}
  \mathcal{V}_{m}=[\V_1,\dots,\V_{m}],
  \qquad
  \mathcal H_{m}=\mathcal{V}_{m}^*A\mathcal{V}_{m}.
\end{equation*}
The matrix $\mathcal H_{m}$ is the recurrence matrix for the basis $\V_1,\dots,\V_{m}$; see~\cref{eq:recurrence-block-arnoldi}.
From the nestedness of the basis, it follows that $\mathcal H_{m}$ is block upper Hessenberg.

The blocks of $\mathcal H_{m}$ can be written explicitly in terms of the Verblunsky coefficients. Combining the two relations in~\cref{eq:block-isometric-arnoldi}, and using the identity $\mrho^R_k\malpha_k=\malpha_k\mrho^L_k$, gives
\begin{equation}
  \label{eq:tilde-v-recursion-schur}
  \widetilde{\V}_{k+1}
  =
  \widetilde{\V}_k\mrho^L_k-\V_{k+1}\malpha_k .
\end{equation}
Hence, by induction from $\widetilde{\V}_1=\V_1$, for $k\geq2$,
\begin{equation}
  \label{eq:tilde-v-expansion-schur}
  \widetilde{\V}_k
  =
  \V_1\mrho^L_1\cdots\mrho^L_{k-1}
  -
  \sum_{h=2}^{k-1}
  \V_h\malpha_{h-1}\mrho^L_h\cdots\mrho^L_{k-1}
  -
  \V_k\malpha_{k-1}.
\end{equation}
Here and below, empty products are interpreted as identity matrices, while empty sums are interpreted as zero.
Substituting~\cref{eq:tilde-v-expansion-schur} into $A\V_k=\V_{k+1}\mrho^R_k+\widetilde{\V}_k\malpha_k^*$ gives, for $k=1,\dots,m$, the expansion
\begin{equation}
  \label{eq:schur-avk-expansion}
  A\V_k
  =
  \V_1\mrho^L_1\cdots\mrho^L_{k-1}\malpha_k^*
  -
  \sum_{h=2}^{k}
  \V_h\malpha_{h-1}\mrho^L_h\cdots\mrho^L_{k-1}\malpha_k^*
  +
  \V_{k+1}\mrho^R_k.
\end{equation}
This is the recurrence relation for the basis $\V_1,\dots,\V_{m}$, up to the residual term in the direction of $\V_{m+1}$ when $k=m$. Its coefficients are exactly the recurrence blocks: if $A\V_k=\sum_i\V_iH_{i,k}$, block orthonormality gives $H_{i,k}=\V_i^*A\V_k$, which is the $(i,k)$ block of $\mathcal H_{m}=\mathcal V_{m}^*A\mathcal V_{m}$. Thus~\cref{eq:schur-avk-expansion} gives
\begin{equation}
  \label{eq:schur-hessenberg-blocks}
  H_{h,k}
  =
  \begin{cases}
    \mrho^L_1\cdots\mrho^L_{k-1}\malpha_k^*,
      & h=1,\\[1mm]
    -\malpha_{h-1}\mrho^L_h\cdots\mrho^L_{k-1}\malpha_k^*,
      & 2\leq h\leq k,\\[1mm]
    \mrho^R_k,
      & h=k+1,\\[1mm]
    0,
      & h\geq k+2.
  \end{cases}
\end{equation}
Hence the projected matrix is completely determined by the Verblunsky coefficients $\malpha_1,\dots,\malpha_{m}$, which play the role of the Schur parameters in the block case, and the associated matrices $\mrho^L_k,\mrho^R_k$. Notice that the last column contains $\malpha_{m}$. Hence, if~\cref{alg:block-isometric-arnoldi} is stopped after producing $\V_1,\dots,\V_{m}$ and $\malpha_1,\dots,\malpha_{m-1}$, one additional multiplication with $A$ and one additional block inner product are needed to recover the final coefficient $\malpha_{m}$.

The same formulas can be summarized more compactly as a Schur parametrization. For $k=1,\dots,m-1$, let
\begin{equation*}
  \mathcal G_k(\malpha_k)
  =
  \begin{bmatrix}
    I_{(k-1)s} \\
    & \malpha_k^* & \mrho^L_k\\
    & \mrho^R_k & -\malpha_k\\
    &&& I_{(m-k-1)s}
  \end{bmatrix},
  \qquad
  \widetilde{\mathcal G}_{m}(\malpha_{m})
  =
  \begin{bmatrix}
    I_{(m-1)s}\\
    & \malpha_{m}^*
  \end{bmatrix}.
\end{equation*}
Then
\begin{equation}
  \label{eq:schur-hessenberg-factorization}
  \mathcal H_{m}
  =
  \mathcal G_1(\malpha_1)
  \mathcal G_2(\malpha_2)
  \cdots
  \mathcal G_{m-1}(\malpha_{m-1})
  \widetilde{\mathcal G}_{m}(\malpha_{m}).
\end{equation}
Each factor $\mathcal G_k(\malpha_k)$ is unitary. If the next Verblunsky coefficient $\malpha_{m}$ is available and satisfies $\|\malpha_{m}\|_2<1$, then the last factor is not unitary in general, but it can be completed to a unitary block factor by adding the corresponding matrices $\mrho^L_{m}$ and $\mrho^R_{m}$. Consequently, $\mathcal H_{m}$ is the projected matrix obtained from $A$ on the block Krylov subspace, but is not itself unitary in general. This factorization is the block analogue of the Schur parametrization of a unitary upper Hessenberg matrix; compare with the scalar isometric Arnoldi setting in~\cite{Gragg93isometric,Watkins93}.

In the single-vector case (i.e., for $s=1$), if one aims to orthogonalize a Krylov subspace associated with a unitary matrix while simultaneously obtaining a projected matrix with a sparse structure, it is more natural to work with extended Krylov subspaces. In that setting, the projected matrix takes the CMV form; see, e.g.,~\cite{Watkins93}. In the next section, we adapt the same idea to extended block Krylov subspaces.
\section{Block CMV matrices and extended block Krylov}
\label{sec:unitary-extended-block-krylov-CMV}

In this section, we exploit the correspondence between orthogonal Laurent matrix polynomials on the unit circle and extended block Krylov subspaces associated with a unitary $A$. We derive a short recurrence relation for the orthogonalization based on block CMV matrices, and turn this strategy into an algorithm.
\subsection{Laurent basis and block CMV matrix}
\label{subsec:cmv}
Here we report the results in~\cite{DaPuSi08} on orthogonal Laurent matrix polynomials and block CMV matrices. 

Let $\{\varphi^R_k\}_{k\geq 0}$, $\{ \varphi^L_k \}_{k\geq 0}$ be sequences of right and left OMPUC, respectively, and consider the reversed left polynomials $\{ \varphi^{L,\#}_k \}_{k\geq 0}$. The \emph{CMV basis} is a sequence $\{\chi_k\}_{k\geq0}$ of Laurent matrix polynomials defined by
\begin{equation}
  \label{eq:cmv-basis}
  \chi_{2k}(z) = z^{-k}\varphi^{L,*}_{2k}(z),\quad \chi_{2k+1}(z) = z^{-k+1} \varphi^R_{2k-1}(z).
\end{equation}
This sequence forms an orthonormal system~\cite[Proposition 3.30]{DaPuSi08}:
\begin{equation*}
  \llangle\chi_i,\chi_j\rrangle_{\bmu} = \delta_{i,j}I_{s}.
\end{equation*}
The functions $\{\chi_k\}_{k\geq 0}$ are used to define the \emph{CMV blocks}:
\begin{equation}
  \label{eq:CMV-blocks-laurent-basis}
  \mathcal{C}_{i,j} = \llangle \chi_{i-1},z\chi_{j-1}\rrangle_{\bmu}.
\end{equation}
Since the $\{\chi_k\}_{k\geq 0}$ are orthonormal, any Laurent matrix polynomial $f$ can be expressed as
\begin{equation*}
  f(z) = \sum_{h\geq 0} \chi_h(z)\llangle \chi_h,f\rrangle_{\bmu}.
\end{equation*}
Hence, for $f(z) = z\chi_k(z)$ we get
\begin{equation}
  \label{eq:CMV-recurrence-functional}
  z\chi_k(z) = \sum_{h\geq 0} \chi_h(z) \llangle \chi_h,z\chi_k\rrangle_{\bmu} = \sum_{h\geq 0} \chi_h(z) \mathcal{C}_{h+1,k+1},
\end{equation} 
and, for $f(z) = z^{-1}\chi_k(z)$,
\begin{equation}
  \label{eq:CMV-recurrence-functional-transposed}
  \begin{aligned}
    z^{-1}\chi_k(z) 
    = 
    \sum_{h\geq 0} \chi_h(z)\llangle \chi_h,z^{-1}\chi_k\rrangle_{\bmu} 
    = 
    \sum_{h\geq 0} \chi_h(z)\llangle z\chi_h,\chi_k\rrangle_{\bmu} 
    = 
    \sum_{h\geq 0} \chi_h(z)\mathcal{C}_{k+1,h+1}^*.
  \end{aligned}
\end{equation}

\begin{remark}
  The CMV basis in~\cref{eq:cmv-basis} is defined in terms of the odd right OMPUC $\varphi^R_{2k+1}$ and the even left reversed polynomials $\varphi^{L,\#}_{2k}$. A related basis depends on $\varphi^R_{2k}$ and $\varphi^{L,\#}_{2k+1}$ instead, and shares important properties with the CMV and the recurrence for the sequence $\{\chi_k\}_{k\geq 0}$. We omit the construction here since it is not relevant for our setting; see~\cite{DaPuSi08} for more details.
\end{remark}

The \emph{block CMV matrix} is the matrix whose blocks are $\mathcal{C}_{i,j}$:
\begin{equation*}
  \mathcal{C} = \begin{bmatrix}
    \mathcal{C}_{1,1}& \mathcal{C}_{1,2}&\cdots\\
    \mathcal{C}_{2,1}& \mathcal{C}_{2,2} & \ddots\\
    \vdots & \ddots&\ddots
  \end{bmatrix}
\end{equation*}
Here, for clarity of exposition, we consider the case of an infinite CMV matrix, which is associated with a non-trivial measure $\bmu$ over the whole space of polynomials (and, in turn, of Laurent polynomials). If the inner product is nondegenerate only on $\poly^s_d$ for a certain $d$, the CMV matrix will be finite. 

The block CMV matrix can be represented $\mathcal{C} = \mathcal{L}\mathcal{M}$, where
\begin{equation*}
  \mathcal{L} = \begin{bmatrix}
    \malpha_1^*&\mrho^L_1\\
    \mrho^R_1&-\malpha_1\\
    &&\malpha_3^*&\mrho^L_3\\
    &&\mrho^R_3&-\malpha_3\\
    &&&&\ddots
  \end{bmatrix},\quad
  \mathcal{M} = \begin{bmatrix}
    I_{s}\\
    &\malpha_2^*&\mrho^L_2\\
    &\mrho^R_2&-\malpha_2\\
    &&&\malpha_4^*&\mrho^L_4\\
    &&&\mrho^R_4&-\malpha_4\\
    &&&&&\ddots
  \end{bmatrix}.
\end{equation*}
Hence, the blocks of $\mathcal{C}$ are completely determined by the Verblunsky coefficients $\{ \malpha_k \}_{k\geq 0}$. Furthermore, most of the CMV blocks are zero. In fact, since $\mathcal{C} = \mathcal{L}\mathcal{M}$,
\begin{equation}
  \label{eq:CMV-block-structure}
  \mathcal{C} = \begin{bmatrix}
    A_0 & B_0\\
    &A_1&B_1\\
    &&A_2&B_2\\
    &&&\ddots&\ddots
  \end{bmatrix}
  = 
  \begin{bmatrix}
    \malpha_1^* & \mrho^L_1\malpha_2^* & \mrho^L_1\mrho^L_2\\
    \mrho^R_1 & -\malpha_1\malpha_2^* & -\malpha_1\mrho^L_2\\
    & \malpha_3^* \mrho^R_2 & -\malpha_3^*\malpha_2 & \mrho^L_3\malpha_4^* & \mrho^L_3 \mrho^L_4\\
    & \mrho^R_3 \mrho^R_2 & -\mrho^R_3 \malpha_2 & -\malpha_3 \malpha_4^* & -\malpha_3 \mrho^L_4\\
    &&& \malpha_5^* \mrho^R_4 & -\malpha_5^*  \malpha_4& \cdots\\
    &&&\vdots&\vdots&\ddots
  \end{bmatrix}
  ,
\end{equation}
where, for $k\geq 1$,
\begin{equation}
  \label{eq:formulas-CMV-blocks-A-B}
  \begin{aligned}
    A_0 
    = 
    \begin{bmatrix}
      \mathcal{C}_{1,1}
      \\ 
      \mathcal{C}_{2,1}
    \end{bmatrix}
    =
    \begin{bmatrix}
      \malpha_1^* \\ \mrho^R_1
    \end{bmatrix},& 
    \qquad 
    A_k 
    = 
    \begin{bmatrix}
    \mathcal{C}_{2k+1,2k} & \mathcal{C}_{2k+1,2k+1}\\
    \mathcal{C}_{2k+2,2k} & \mathcal{C}_{2k+2,2k+1}
    \end{bmatrix}
    =
    \begin{bmatrix}
      \malpha_{2k+1}^* \mrho^R_{2k} & -\malpha_{2k+1}^*\malpha_{2k}\\
      \mrho^R_{2k+1} \mrho^R_{2k} & -\mrho^R_{2k+1}\malpha_{2k}
    \end{bmatrix},\\
    B_{k-1} 
    =& 
    \begin{bmatrix}
      \mathcal{C}_{2k-1,2k} & \mathcal{C}_{2k-1,2k+1}\\
      \mathcal{C}_{2k,2k} & \mathcal{C}_{2k,2k+1}
    \end{bmatrix}
    =
    \begin{bmatrix}
      \mrho^L_{2k-1} \malpha_{2k}^* & \mrho^L_{2k-1} \mrho^L_{2k}\\
      -\malpha_{2k-1}\malpha_{2k}^* & -\malpha_{2k-1}\mrho^L_{2k}
    \end{bmatrix}.
  \end{aligned}
\end{equation}
Thus, in~\cref{eq:CMV-recurrence-functional} and~\cref{eq:CMV-recurrence-functional-transposed} we actually have two short recurrence relations involving $z\chi_k(z)$ and $z^{-1}\chi_k(z)$, respectively.

\subsection{CMV structure and extended block Krylov subspaces}

We first clarify the finite-degree issue. If an inner product supported on the unit circle is nondegenerate on $\poly^s_d$, then the finite right and left OMPUC $\varphi^R_0,\dots,\varphi^R_d$ and $\varphi^L_0,\dots,\varphi^L_d$ are well defined, together with the Verblunsky coefficients $\malpha_1,\dots,\malpha_d$. The Szeg\H{o} recurrence is then valid for $k=0,\dots,d-1$, so all finite formulas involving only these coefficients are intrinsic to the original inner product. In the Krylov setting the induced inner product is always degenerate on the full space $\poly^s$; see~\cref{rem:discrete-inner-product-degenerate}. In order to use the standard infinite CMV notation and sparsity pattern without truncating every formula separately, we prove the following result.
\begin{proposition}
  \label{prop:infinite-extension-OMPUC}
  Let $\bmu$ be a Hermitian matrix-valued measure supported on $\mathbb{T}$ such that the inner product $\llangle\cdot,\cdot\rrangle_{\bmu}$ is nondegenerate on $\poly^s_{d}$, for a positive integer $d$. Then, there exists a measure $\widetilde{\bmu}$ on $\mathbb{T}$ such that $\llangle\cdot,\cdot\rrangle_{\widetilde{\bmu}}$ is nondegenerate over $\poly^s$ and
  \begin{equation}
    \label{eq:extended-inner-product-unit-circle}
    \llangle P,Q\rrangle_{\bmu} 
    = 
    \llangle P,Q \rrangle_{\widetilde{\bmu}}
  \end{equation}
  for all $P,Q\in\poly^s_{d}$.
\end{proposition}

\begin{proof}
  Since $\llangle \cdot,\cdot\rrangle_{\bmu}$ is supported on the unit circle and nondegenerate over $\poly^s_d$, there exist finite sequences of right OMPUC $\varphi^R_0,\dots,\varphi^R_d$ and left OMPUC $\varphi^L_0,\dots,\varphi^L_d$. The corresponding Szeg\H{o} recurrence holds for $k=0,\dots,d-1$, with Verblunsky coefficients $\malpha_k$ given by~\cref{eq:verblunsky-formula-OMPUC} for $k=1,\dots,d$.
  By Verblunsky's Theorem~\cite[Theorem 3.12]{DaPuSi08}, infinite sequences of Verblunsky coefficients are in one-to-one correspondence with nondegenerate matrix-valued measures on $\mathbb{T}$. Hence, we define $\{\widetilde{\malpha}_k\}_{k\geq 1}$ by setting $\widetilde{\malpha}_k=\malpha_k$ for $k=1,\dots,d$, and by choosing arbitrary matrices satisfying $\|\widetilde{\malpha}_k\|_2<1$ for $k\geq d+1$. We then let $\widetilde{\bmu}$ be the matrix-valued measure associated with $\{\widetilde{\malpha}_k\}_{k\geq 1}$.

  We now show that~\cref{eq:extended-inner-product-unit-circle} holds. Since the first $d$ Verblunsky coefficients coincide, then $\varphi^R_0,\dots,\varphi^R_d$ are OMPUC for both inner products. Therefore, for any $P,Q\in\poly^s_d$, we can write
  \begin{equation*}
      P(z) = \sum_{k=0}^{d} \varphi^R_k C_k,
      \qquad 
      Q(z) = \sum_{k=0}^{d} \varphi^R_k D_k,
  \end{equation*} 
  for $C_k,D_k\in\C^{s\times s}$. Using orthonormality with respect to both inner products gives
  \begin{equation*}
    \llangle P,Q\rrangle_{\bmu} = \sum_{k=0}^{d} C_k^* D_k = \llangle P,Q\rrangle_{\widetilde{\bmu}}.
  \end{equation*}
  This concludes the proof.
\end{proof}

In the Krylov application below, we use~\cref{prop:infinite-extension-OMPUC} with $d=m-1$, since full dimension of the relevant Krylov space gives nondegeneracy on the polynomial space of degree at most $m-1$. This allows us to use the standard infinite CMV notation and block formulas, while only retaining the finite initial cutoff needed in the Krylov construction.

We now apply the theory of CMV bases and matrices to the orthogonalization of extended block Krylov subspaces $\ebkryl_m(A,\B)$. Note that, since $A$ is unitary, the negative powers appearing in~\cref{eq:extended-block-krylov-blockspan} can be replaced by powers of $A^*$.

From~\cref{thm:isomorphism-extended-laurent} the inner product $\llangle\cdot,\cdot\rrangle_{A,\B}$ is nondegenerate over $\poly^s_{m-1}$ if and only if $\ebkryl_{m}(A,\B)$ has full dimension, in which case orthonormal bases are in correspondence. We consider the CMV basis $\{\chi_k\}_{k\geq 0}$ associated with $\llangle\cdot,\cdot\rrangle_{A,\B}$, and define the block vectors
\begin{equation*}
  \W_k = \chi_{k-1}(A)\circ\B,\quad k=1,\dots,m.
\end{equation*}
Because of the isometry, the $\W_k$ are orthonormal and form a basis of $\ebkryl_m(A,\B)$. Moreover, from~\cref{eq:CMV-blocks-laurent-basis}, the CMV blocks coincide with
\begin{equation}
  \label{eq:CMV-blocks-krylov-basis}
  \mathcal{C}_{i,j} = \W_i^* A \W_j.
\end{equation}

We now exploit the sparsity pattern in the CMV matrix to derive recurrence relations for the sequence $\{\W_k\}_{k\geq 1}$ and an orthogonalization procedure for $\ebkryl_{m}(A,\B)$. Since $(z\chi_{k-1})(A)\circ\B = A\W_k$ and $(z^{-1}\chi_{k-1})(A)\circ\B = A^*\W_k$, the relations~\cref{eq:CMV-recurrence-functional} and~\cref{eq:CMV-recurrence-functional-transposed} are transformed as follows:
\begin{align}
  \label{eq:recurrence-w1-A}
  A\W_1 &= \W_1\mathcal{C}_{1,1} + \W_2 \mathcal{C}_{2,1},\\
  \label{eq:recurrence-w1-A*}
  A^* \W_1 &= \W_1 \mathcal{C}_{1,1}^* + \W_2 \mathcal{C}_{1,2}^* + \W_3 \mathcal{C}_{1,3}^*,\\
  A^* \W_2 &= \W_1 \mathcal{C}_{2,1}^* + \W_2 \mathcal{C}_{2,2}^* + \W_3 \mathcal{C}_{2,3}^*,
\end{align}
and
{
\small
\begin{align}
  \label{eq:CMV-recursion-odd-transpose}
  A^* \W_{2k-1} &= \W_{2k-2} \mathcal{C}_{2k-1,2k-2}^* + \W_{2k-1} \mathcal{C}_{2k-1,2k-1}^* + \W_{2k} \mathcal{C}_{2k-1,2k}^* + \W_{2k+1} \mathcal{C}_{2k-1,2k+1}^*,\\
  \label{eq:CMV-recursion-even-transpose}
  A^* \W_{2k} &= \W_{2k-2} \mathcal{C}_{2k,2k-2}^* + \W_{2k-1} \mathcal{C}_{2k,2k-1}^* + \W_{2k} \mathcal{C}_{2k,2k}^* + \W_{2k+1} \mathcal{C}_{2k,2k+1}^*,\\
  \label{eq:CMV-recursion-even-regular}
  A\W_{2k} &= \W_{2k-1} \mathcal{C}_{2k-1,2k} + \W_{2k} \mathcal{C}_{2k,2k} + \W_{2k+1} \mathcal{C}_{2k+1,2k} + \W_{2k+2} \mathcal{C}_{2k+2,2k},\\
  \label{eq:CMV-recursion-odd-regular}
  A\W_{2k+1} &= \W_{2k-1} \mathcal{C}_{2k-1,2k+1} + \W_{2k} \mathcal{C}_{2k,2k+1} + \W_{2k+1} \mathcal{C}_{2k+1,2k+1} + \W_{2k+2} \mathcal{C}_{2k+2,2k+1}.
\end{align}
}
Note that the nonzero CMV blocks can be expressed in terms of the Verblunsky coefficients $\{\malpha_k\}_{k\geq1}$. We use the representation~\cref{eq:formulas-CMV-blocks-A-B}.
Since
\begin{equation*}
  \begin{bmatrix}
      \mathcal{C}_{1,1}
      \\ 
      \mathcal{C}_{2,1}
    \end{bmatrix}
    =
    \begin{bmatrix}
      \malpha_1^* \\ \mrho^R_1
    \end{bmatrix},
\end{equation*}
 $\mathcal{C}_{1,1}$ and $\mathcal{C}_{2,1}$ depend only on $\malpha_1$. From
\begin{equation}
  \label{eq:CMV-blocks-odd}
  \begin{bmatrix}
      \mathcal{C}_{2k-1,2k} & \mathcal{C}_{2k-1,2k+1}\\
      \mathcal{C}_{2k,2k} & \mathcal{C}_{2k,2k+1}
    \end{bmatrix}
    =
    \begin{bmatrix}
      \mrho^L_{2k-1} \malpha_{2k}^* & \mrho^L_{2k-1} \mrho^L_{2k}\\
      -\malpha_{2k-1}\malpha_{2k}^* & -\malpha_{2k-1}\mrho^L_{2k}
    \end{bmatrix},
\end{equation}
we get that $\mathcal{C}_{2k-1,2k}$, $\mathcal{C}_{2k-1,2k+1}$, $\mathcal{C}_{2k,2k}$, and $\mathcal{C}_{2k,2k+1}$ depend only on $\malpha_{2k-1}$ and $\malpha_{2k}$, and from
 \begin{equation}
  \label{eq:CMV-blocks-even}
  \begin{bmatrix}
    \mathcal{C}_{2k+1,2k} & \mathcal{C}_{2k+1,2k+1}\\
    \mathcal{C}_{2k+2,2k} & \mathcal{C}_{2k+2,2k+1}
    \end{bmatrix}
    =
    \begin{bmatrix}
      \malpha_{2k+1}^* \mrho^R_{2k} & -\malpha_{2k+1}^*\malpha_{2k}\\
      \mrho^R_{2k+1} \mrho^R_{2k} & -\mrho^R_{2k+1}\malpha_{2k}
    \end{bmatrix},
 \end{equation}
 we get that 
 $\mathcal{C}_{2k+1,2k}$, $\mathcal{C}_{2k+1,2k+1}$, $\mathcal{C}_{2k+2,2k}$, and $\mathcal{C}_{2k+2,2k+1}$ depend on $\malpha_{2k}$ and $\malpha_{2k+1}$.
Together, these identities give all nonzero CMV blocks.

\subsection{An orthogonalization procedure for $\ebkryl_m(A,\B)$}

We now use the formulas for the CMV blocks to obtain a procedure for the sequential computation of $\W_k$, $\mathcal{C}_{i,j}$, and $\malpha_k$. At each iteration, we want to perform just one multiplication with either $A$ or $A^*$.

\paragraph{Initialize} If $\B^*\B=I_s$, set $\W_1=\B$. Otherwise, since we assume that $\llangle\cdot,\cdot\rrangle_{A,\B}$ is nondegenerate on $\poly^{s}_{m-1}$, we can choose $\W_1$ so that $\B = \W_1R$, where $R\in\C^{s\times s}$ and $\W_1^*\W_1=I_{s}$.

\paragraph{First step} To get $\W_2$, we exploit~\cref{eq:recurrence-w1-A}, since~\cref{eq:recurrence-w1-A*} also involves $\W_3$. We then compute $\X_{2}=A\W_1$ and $\mathcal{C}_{1,1} = \W_1^* \X_2$. We cannot compute $\mathcal{C}_{2,1}$ via~\cref{eq:CMV-blocks-krylov-basis}, since $\W_2$ is not yet available. We instead retrieve $\malpha_1=\mathcal{C}_{1,1}^*$ and $\mrho^L_1$, $\mrho^R_1$ from~\cref{eq:rho-formula}, so we set $\mathcal{C}_{2,1} = \mrho^R_1$. Recall that $\mrho^R_k$ and $\mrho^L_k$ are always nonsingular; see~\cref{sec:unitary-block-Krylov}. Hence, we compute $\W_2$ from~\cref{eq:recurrence-w1-A}.

\paragraph{Next steps} To compute $\W_3,\W_4,\dots$ we follow the same idea: perform a multiplication with either $A$ or $A^*$, compute via an inner product a CMV block that allows to extract the next Verblunsky coefficient, then form the other CMV blocks needed for the recurrence relation. Because of the structure of the CMV matrix and of the recurrence relations~\crefrange{eq:CMV-recursion-odd-transpose}{eq:CMV-recursion-odd-regular}, it is convenient to distinguish between odd and even iterations.

\paragraph{Odd iteration} Assume that $\W_1,\dots,\W_{2k}$ and $\malpha_1,\dots,\malpha_{2k-1}$ are available from previous iterations, together with the corresponding CMV blocks. We can use either~\cref{eq:CMV-recursion-odd-transpose} or~\cref{eq:CMV-recursion-even-transpose} to compute $\W_{2k+1}$. We opt for~\cref{eq:CMV-recursion-odd-transpose}: the coefficient $\malpha_{2k}$ appears in $\mathcal{C}_{2k-1,2k} = \mrho^L_{2k-1}\malpha_{2k}^*$ and in $\mathcal{C}_{2k,2k}=-\malpha_{2k-1}\malpha_{2k}^*$. However, $\malpha_{2k-1}$ is not guaranteed to be nonsingular, so inverting $\mathcal{C}_{2k,2k}$ may not be possible. Thus, we compute $\X_{2k+1} = A^*\W_{2k-1}$, compute $\mathcal{C}_{2k-1,2k}= \X_{2k+1}^*\W_{2k}$ and $\malpha_{2k} = \mathcal{C}_{2k-1,2k}^*(\mrho^L_{2k-1})^{-1}$. At this point, we can form all the CMV blocks in~\cref{eq:CMV-blocks-odd}. Since $\mathcal{C}_{2k-1,2k-2}$ and $\mathcal{C}_{2k-1,2k-1}$ are available from previous iterations, we can use the relation~\cref{eq:CMV-recursion-odd-transpose} to compute $\W_{2k+1}$. Note that $\mathcal{C}_{2k-1,2k+1}=\mrho^L_{2k-1}\mrho^L_{2k}$ is nonsingular.

\paragraph{Even iteration} Assume that we have $\W_1,\dots,\W_{2k+1}$ and $\malpha_1,\dots,\malpha_{2k}$, and we want to compute $\W_{2k+2}$. Similarly as before, we choose~\cref{eq:CMV-recursion-even-regular} instead of~\cref{eq:CMV-recursion-odd-regular}, since $\malpha_{2k+1}$ can be extracted from $\mathcal{C}_{2k+1,2k} = \malpha_{2k+1}^*\mrho^R_{2k}$, and $\mathcal{C}_{2k+2,2k}=\mrho^R_{2k+1}\mrho^R_{2k}$ is necessarily nonsingular. Hence, we compute $\X_{2k+2} = A\W_{2k}$, $\mathcal{C}_{2k+1,2k}=\W_{2k+1}^*\X_{2k+2}$, and $\malpha_{2k+1}=(\mrho^R_{2k})^{-1} \mathcal{C}_{2k+1,2k}^*$. Then we form all the CMV blocks in~\cref{eq:CMV-blocks-even}, and use the relation~\cref{eq:CMV-recursion-even-regular} to compute $\W_{2k+2}$.

\begin{algorithm}[ht]
\caption{CMV based orthogonalization of extended block Krylov subspaces}
\label{alg:block-CMV-arnoldi}
\begin{algorithmic}[1]
\REQUIRE Unitary $A\in\C^{n\times n}$, starting block vector $\B\in\C^{n\times s}$, $\B^*\B=I_{s}$, number of steps $m$
\ENSURE Orthonormal basis $\vec W_1,\dots,\vec W_{m}$ of $\ebkryl_{m}(A,\B)$, CMV blocks $\mathcal{C}_{i,j}$, and Verblunsky coefficients $\malpha_1,\dots,\malpha_{m-1}$.

\STATE \textbf{Initialize.} Set $\W_0=\boldsymbol 0$, $\mathcal{C}_{1,0}= \vec 0$, $\W_1 = \B$
\STATE \textbf{Multiplication with $A$.} Compute $\X_2 = A\W_1$ and $\mathcal{C}_{1,1} = \W_1^*\X_2$ 
\STATE Set $\malpha_1 = \mathcal{C}_{1,1}^*$, form $\mrho^R_1 = (I_{s}-\malpha_1\malpha_1^*)^{\frac{1}{2}}$, $\mrho^L_1 = (I_{s}-\malpha_1^*\malpha_1)^{\frac{1}{2}}$, $\mathcal{C}_{2,1} = \mrho^R_1$
\STATE Compute $\Y_2 = \X_2 - \W_1 \mathcal{C}_{1,1}$ and $\W_2 = \Y_2 \mathcal{C}_{2,1}^{-1}$
\FOR{$j=3,4,\dots,m$}
  \IF{$j=2k+1$}
    \STATE \textbf{Multiplication with $A^*$.} Compute $\X_{2k+1} = A^*\W_{2k-1}$
    \STATE Compute $\mathcal{C}_{2k-1,2k} = \X_{2k+1}^* \W_{2k}$, form $\malpha_{2k} = \mathcal{C}_{2k-1,2k}^* (\mrho^L_{2k-1})^{-1}$,\\ $\mrho^R_{2k} = (I_{s}-\malpha_{2k}\malpha_{2k}^*)^{\frac{1}{2}}$, $\mrho^L_{2k} = (I_{s}-\malpha_{2k}^*\malpha_{2k})^{\frac{1}{2}}$
    \STATE Form $\mathcal{C}_{2k-1,2k+1} = \mrho^L_{2k-1} \mrho^L_{2k}$, $\mathcal{C}_{2k,2k} = -\malpha_{2k-1}\malpha_{2k}^*$, $\mathcal{C}_{2k,2k+1} = -\malpha_{2k-1}\mrho^L_{2k}$
    \STATE Compute $\Y_{2k+1} = \X_{2k+1} - \W_{2k-2} \mathcal{C}_{2k-1,2k-2}^* - \W_{2k-1} \mathcal{C}_{2k-1,2k-1}^* - \W_{2k} \mathcal{C}_{2k-1,2k}^*$\\
    and $\W_{2k+1} =\Y_{2k+1}(\mathcal{C}_{2k-1,2k+1}^*)^{-1}$
    \label{line:linear-combinations-odd}
    \ENDIF
  \IF{$j=2k+2$}
    \STATE \textbf{Multiplication with $A$.} $\X_{2k+2} = A \W_{2k}$ and $\mathcal{C}_{2k+1,2k} = \W_{2k+1}^*\X_{2k+2}$
    \STATE $\malpha_{2k+1} = (\mrho^R_{2k})^{-1}\mathcal{C}_{2k+1,2k}^*$, $\mrho^R_{2k+1} = (I_{s} - \malpha_{2k+1}\malpha_{2k+1}^*)^{\frac{1}{2}}$, $\mrho^L_{2k+1} = (I_{s} - \malpha_{2k+1}^*\malpha_{2k+1})^{\frac{1}{2}}$
    \STATE Form $\mathcal{C}_{2k+1,2k+1} = -\malpha_{2k+1}^*\malpha_{2k}$, $\mathcal{C}_{2k+2,2k} = \mrho^R_{2k+1} \mrho^R_{2k}$, $\mathcal{C}_{2k+2,2k+1} = -\mrho^R_{2k+1}\malpha_{2k}$
    \STATE Compute $\Y_{2k+2} = \X_{2k+2} - \W_{2k-1} \mathcal{C}_{2k-1,2k} - \W_{2k} \mathcal{C}_{2k,2k} - \W_{2k+1} \mathcal{C}_{2k+1,2k}$\\ and $\W_{2k+2} = \Y_{2k+2}\mathcal{C}_{2k+2,2k}^{-1}$
    \label{line:linear-combinations-even}
  \ENDIF
\ENDFOR
\end{algorithmic}
\end{algorithm}
The procedure is summarized in~\cref{alg:block-CMV-arnoldi}.
We now analyze the computational cost. At each iteration, we perform a matrix-block-vector multiplication with $A$ or $A^*$, whose cost depends on the structure of $A$. We perform one matrix-valued Euclidean inner product, which costs $O(s^2n)$. At lines~\ref{line:linear-combinations-odd} and~\ref{line:linear-combinations-even}, we perform a linear combination of block vectors, which again costs $O(s^2n)$.

\cref{alg:block-CMV-arnoldi} returns the basis $\W_1,\dots,\W_m$, the Verblunsky coefficients $\malpha_1,\dots,\malpha_{m-1}$, and the CMV blocks $\mathcal{C}_{i,j}$. By~\cref{eq:CMV-blocks-krylov-basis}, the orthogonal projection of $A$ onto $\ebkryl_m(A,\B)$ associated with this basis is the \emph{cutoff block CMV matrix}\footnote{Cutoff CMV matrices are widely studied for their connections with the roots of orthogonal polynomials and matrix polynomials on the unit circle; see, e.g.,~\cite{DaPuSi08}. }
\begin{equation*}
  [\W_1,\dots,\W_m]^*\, A \,[\W_1,\dots,\W_m] = \begin{bmatrix}
    \mathcal{C}_{1,1} & \mathcal{C}_{1,2} & \cdots & \mathcal{C}_{1,m}\\
    \mathcal{C}_{2,1} & \mathcal{C}_{2,2} & \cdots & \mathcal{C}_{2,m}\\
    \vdots&\vdots &\ddots & \vdots\\
    \mathcal{C}_{m,1} & \mathcal{C}_{m,2} & \cdots & \mathcal{C}_{m,m}
  \end{bmatrix}=\mathcal{C}_m.
\end{equation*}
 However, not every block required for $\mathcal{C}_m$ is available after the execution of~\cref{alg:block-CMV-arnoldi}, since the Verblunsky coefficient $\malpha_m$ has not been computed. The missing blocks are $\mathcal{C}_{m-1,m}$ and $\mathcal{C}_{m,m}$ when $m$ is odd, and $\mathcal{C}_{m,m-1}$ and $\mathcal{C}_{m,m}$ when $m$ is even. We can retrieve $\malpha_m$ and the missing blocks through an additional multiplication with $A$ or $A^*$, and the evaluation of two inner products.

 The projected matrices described here and in~\cref{sec:unitary-block-Krylov} contain the same spectral information.
\begin{proposition}
\label{prop:polynomial-extended-projected-similarity}
Let $A$ be unitary, and let
\begin{equation*}
    \mathcal{V}_m = [\V_1,\dots,\V_m],
    \quad 
    \mathcal{W}_m = [\W_1,\dots,\W_m],
\end{equation*}
be the block orthonormal bases of $\bkryl_m(A,\B)$ and $\ebkryl_m(A,\B)$, generated by~\cref{alg:block-isometric-arnoldi} and~\cref{alg:block-CMV-arnoldi}, respectively, and let
\begin{equation*}
    \mathcal{H}_m = \mathcal{V}_m^*A\mathcal{V}_m,
    \quad
    \mathcal{C}_m = \mathcal{W}_m^*A\mathcal{W}_m.
\end{equation*}
Then, there is a unitary $S_m\in\C^{ms\times ms}$ such that
\begin{equation*}
    \mathcal{C}_m = S_m^* \mathcal{H}_m S_m.
\end{equation*}
\end{proposition}
\begin{proof}
Let $\ell=\lfloor (m-1)/2\rfloor$ and set
\begin{equation*}
    \mathcal Z_m=A^\ell \mathcal W_m .
\end{equation*}
By~\cref{eq:extended-polynomial-block-krylov-correspondence}, the block columns of $\mathcal Z_m$ form an orthonormal basis of $\bkryl_m(A,\B)$. Hence there is a unitary matrix $S_m\in\C^{ms\times ms}$ such that
\begin{equation*}
    \mathcal Z_m=\mathcal V_m S_m,
    \qquad S_m=\mathcal V_m^*\mathcal Z_m .
\end{equation*}
Since $A$ is unitary and commutes with its powers,
\begin{equation*}
    \mathcal C_m
    =\mathcal W_m^*A\mathcal W_m
    =\mathcal Z_m^*A\mathcal Z_m 
    = S_m^*\mathcal V_m^*A\mathcal V_m S_m
    = S_m^*\mathcal H_m S_m.
\end{equation*}
This concludes the proof.
\end{proof}

\section{Numerical experiments}
\label{sec:numerical-experiments}

In this section we test the orthogonalization procedures derived in \cref{alg:block-isometric-arnoldi} and \cref{alg:block-CMV-arnoldi}. We first check the correctness and efficiency of the algorithms, and then test the methods for computing eigenvalues of multiplicity higher than one, which is a classical application of block Krylov methods~\cite{BK-multiple-eigenvalues}. All numerical experiments are performed in MATLAB on a laptop equipped with an AMD Ryzen 7 PRO 7840U processor and 16 GB of RAM. The code used in the experiments is fully available at~\url{https://numa.cs.kuleuven.be/software/momentum-software}.

\subsection{\texorpdfstring{Efficient orthogonalization}{Efficient orthogonalization}}

Let $\mathcal X_m=[\vec X_1,\dots,\vec X_m]$ denote the block basis produced by one of the algorithms, and let $\mathcal T_m$ denote the corresponding recurrence matrix. For standard block Arnoldi this is the block Hessenberg matrix produced by the algorithm; for block isometric Arnoldi it is the block Hessenberg matrix obtained from the Schur parametrization in~\cref{eq:schur-hessenberg-blocks}; for block CMV Arnoldi it is the cutoff block CMV matrix. We measure the loss of orthogonality by
\begin{equation}
\label{eq:numerical-orthogonality-error}
    \varepsilon_{\mathrm{orth}}
    = \|\mathcal X_m^*\mathcal X_m-I_{ms}\|_F,
\end{equation}
and the accuracy of the projection by
\begin{equation}
\label{eq:numerical-projection-error}
    \varepsilon_{\mathrm{proj}}
    =
    \|\mathcal X_m^*A\mathcal X_m-\mathcal T_m\|_F.
\end{equation}

For both the block isometric Arnoldi and block CMV Arnoldi cases, forming $\mathcal T_m$ requires the coefficient $\malpha_m$. Thus, for this accuracy check only, we compute $\malpha_m$ by one additional multiplication and one additional block inner product. This extra step is not included in the timing experiment. 

The reference method for the polynomial space $\bkryl_m(A,\B)$ is the standard block Arnoldi process in~\cref{alg:block-arnoldi}. For the extended space $\ebkryl_m(A,\B)$, the reference method is block extended Arnoldi; see~\cite{extended-block-arnoldi}.

In this experiment we take $n=100\,000$ and $s=10$, and vary the number of steps $m$. The test matrix is a sparse Floquet matrix of size $n$. More precisely, for a sequence $\alpha_1,\dots,\alpha_n$ with $|\alpha_j|<1$ we set $\rho_j=(1-|\alpha_j|^2)^{1/2}$ and
\begin{equation*}
\Theta(\alpha_j)=
\begin{bmatrix}
\overline{\alpha_j} & \rho_j\\
\rho_j & -\alpha_j
\end{bmatrix}.
\end{equation*}
The unitary matrix $A$ is then given by $A=F_1F_2$, where
\begin{equation*}
F_1=
\begin{bmatrix}
\Theta(\alpha_1)\\
&\Theta(\alpha_3)\\
&&\ddots\\
&&&\Theta(\alpha_{n-1})
\end{bmatrix},
\qquad
F_2=
\begin{bmatrix}
\overline{\alpha_n}&&&&\rho_n\\
&\Theta(\alpha_2)\\
&&\ddots\\
&&&\Theta(\alpha_{n-2})\\
\rho_n&&&&-\alpha_n
\end{bmatrix}.
\end{equation*}
The eigenvalues of $A$ correspond to the zeros of Szeg\H{o} transfer-matrix discriminants, which in turn are related to Lee-Yang zeros for one-dimensional ferromagnetic Ising models; see~\cite{EmbreeFloquet24}. For our tests, the parameters $\alpha_j$ are sampled randomly in the open unit disk. Both factors are sparse, and so is their product. Hence, a multiplication of $A$ or $A^*$ by an $n\times s$ block vector costs $O(ns)$. The timings in~\cref{fig:efficient-orthogonalization-timing} include only the construction of the basis vectors. In particular, they do not include the explicit formation of the recurrence matrix. The corresponding orthogonality and projection errors are shown in~\cref{fig:efficient-orthogonalization-accuracy}.

\begin{figure}[ht]
\centering
\begin{tikzpicture}
\node[draw, fill=white, inner xsep=4pt, inner ysep=2pt] {\scriptsize
\begin{tabular}{@{}c@{\ }l@{\quad}c@{\ }l@{\quad}c@{\ }l@{\quad}c@{\ }l@{}}
\tikz[baseline=-0.6ex]{\draw[blue, thick] (0,0)--(0.45,0); \draw[blue, thick, fill=white] (0.225,0) circle[radius=1.4pt];} & block Arnoldi &
\tikz[baseline=-0.6ex]{\draw[orange!80!black, thick] (0,0)--(0.45,0); \draw[orange!80!black, thick, fill=white] (0.18,-0.045) rectangle (0.27,0.045);} & block isometric Arnoldi &
\tikz[baseline=-0.6ex]{\draw[green!60!black, thick] (0,0)--(0.45,0); \draw[green!60!black, thick, fill=white] (0.225,0.05)--(0.18,-0.04)--(0.27,-0.04)--cycle;} & block extended Arnoldi &
\tikz[baseline=-0.6ex]{\draw[red!75!black, thick] (0,0)--(0.45,0); \draw[red!75!black, thick, fill=white] (0.225,0.055)--(0.275,0)--(0.225,-0.055)--(0.175,0)--cycle;} & block CMV Arnoldi
\end{tabular}};
\end{tikzpicture}
\vspace{0.4em}

\begin{tikzpicture}
\begin{semilogyaxis}[
    width=0.86\textwidth,
    height=0.43\textwidth,
    grid=both,
    xlabel={$m$},
    ylabel={time for basis construction (s)},
    tick label style={font=\footnotesize},
    label style={font=\small},
]
\addplot+[blue, mark=o, mark options={solid, fill=white}, thick] table[x=m,y=time_arnoldi] {efficient_orthogonalization_partial.dat};
\addplot+[orange!80!black, mark=square, mark options={solid, fill=white}, thick] table[x=m,y=time_isometric] {efficient_orthogonalization_partial.dat};
\addplot+[green!60!black, mark=triangle, mark options={solid, fill=white}, thick] table[x=m,y=time_laurent_gs] {efficient_orthogonalization_partial.dat};
\addplot+[red!75!black, mark=diamond, mark options={solid, fill=white}, thick] table[x=m,y=time_cmv] {efficient_orthogonalization_partial.dat};
\end{semilogyaxis}
\end{tikzpicture}
\caption{CPU time for constructing the basis vectors, as a function of the number of steps $m$. The block size is $s=10$ and the matrix size is $n=100\,000$.}
\label{fig:efficient-orthogonalization-timing}
\end{figure}
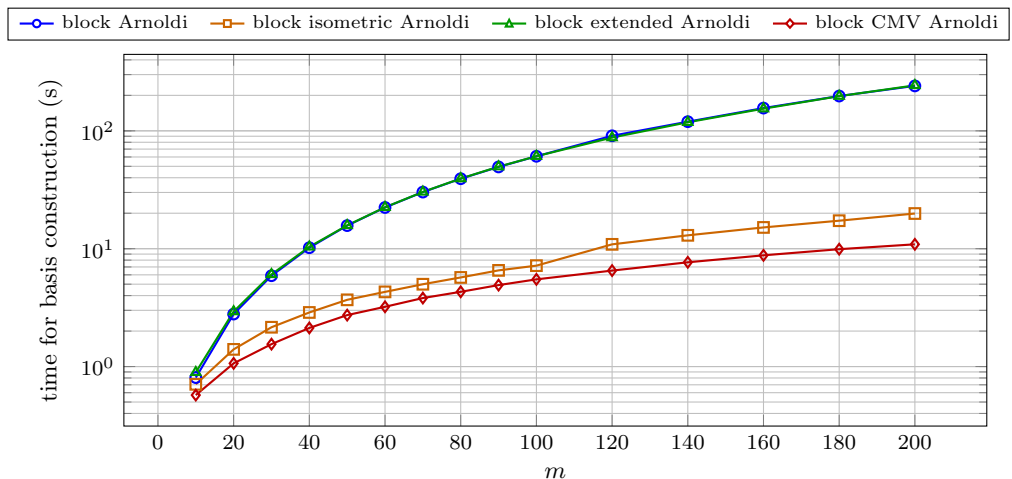

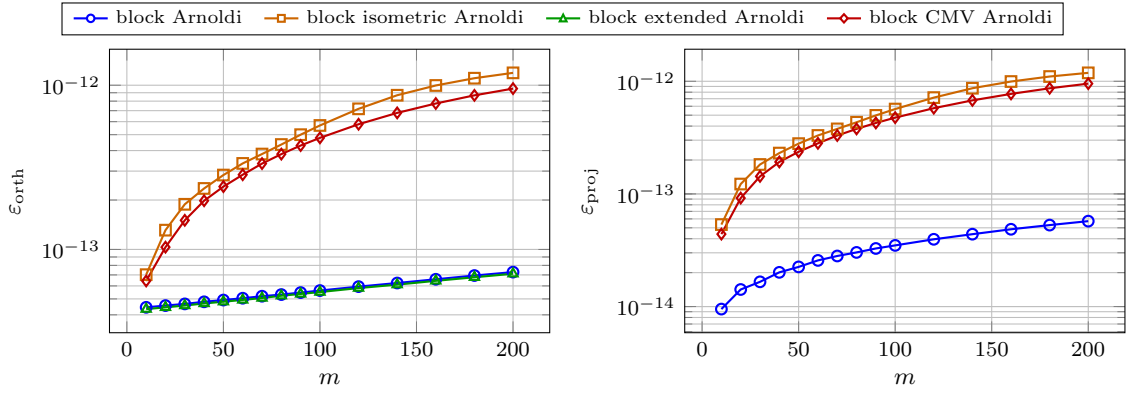
\begin{figure}[ht]
\centering
\begin{tikzpicture}
\node[draw, fill=white, inner xsep=4pt, inner ysep=2pt] {\scriptsize
\begin{tabular}{@{}c@{\ }l@{\quad}c@{\ }l@{\quad}c@{\ }l@{\quad}c@{\ }l@{}}
\tikz[baseline=-0.6ex]{\draw[blue, thick] (0,0)--(0.45,0); \draw[blue, thick, fill=white] (0.225,0) circle[radius=1.4pt];} & block Arnoldi &
\tikz[baseline=-0.6ex]{\draw[orange!80!black, thick] (0,0)--(0.45,0); \draw[orange!80!black, thick, fill=white] (0.18,-0.045) rectangle (0.27,0.045);} & block isometric Arnoldi &
\tikz[baseline=-0.6ex]{\draw[green!60!black, thick] (0,0)--(0.45,0); \draw[green!60!black, thick, fill=white] (0.225,0.05)--(0.18,-0.04)--(0.27,-0.04)--cycle;} & block extended Arnoldi &
\tikz[baseline=-0.6ex]{\draw[red!75!black, thick] (0,0)--(0.45,0); \draw[red!75!black, thick, fill=white] (0.225,0.055)--(0.275,0)--(0.225,-0.055)--(0.175,0)--cycle;} & block CMV Arnoldi
\end{tabular}};
\end{tikzpicture}
\vspace{0.4em}

\begin{minipage}{0.49\textwidth}
\centering
\begin{tikzpicture}
\begin{semilogyaxis}[
    width=\linewidth,
    height=0.72\linewidth,
    grid=both,
    xlabel={$m$},
    ylabel={$\varepsilon_{\mathrm{orth}}$},
    tick label style={font=\footnotesize},
    label style={font=\small},
]
\addplot+[blue, mark=o, thick] table[x=m,y=orth_arnoldi] {efficient_orthogonalization_partial.dat};
\addplot+[orange!80!black, mark=square, mark options={solid, fill=white}, thick] table[x=m,y=orth_isometric] {efficient_orthogonalization_partial.dat};
\addplot+[green!60!black, mark=triangle, mark options={solid, fill=white}, thick] table[x=m,y=orth_laurent_gs] {efficient_orthogonalization_partial.dat};
\addplot+[red!75!black, mark=diamond, mark options={solid, fill=white}, thick] table[x=m,y=orth_cmv] {efficient_orthogonalization_partial.dat};
\end{semilogyaxis}
\end{tikzpicture}
\end{minipage}
\hfill
\begin{minipage}{0.49\textwidth}
\centering
\begin{tikzpicture}
\begin{semilogyaxis}[
    width=\linewidth,
    height=0.72\linewidth,
    grid=both,
    xlabel={$m$},
    ylabel={$\varepsilon_{\mathrm{proj}}$},
    tick label style={font=\footnotesize},
    label style={font=\small},
]
\addplot+[blue, mark=o, thick] table[x=m,y=proj_arnoldi] {efficient_orthogonalization_partial.dat};
\addplot+[orange!80!black, mark=square, mark options={solid, fill=white}, thick] table[x=m,y=proj_isometric] {efficient_orthogonalization_partial.dat};
\addplot+[red!75!black, mark=diamond, mark options={solid, fill=white}, thick] table[x=m,y=proj_cmv] {efficient_orthogonalization_partial.dat};
\end{semilogyaxis}
\end{tikzpicture}
\end{minipage}
\caption{Loss of orthogonality and projection error for the computed bases.}
\label{fig:efficient-orthogonalization-accuracy}
\end{figure}

For block isometric Arnoldi, the projection error compares $\mathcal X_m^*A\mathcal X_m$ with the block Hessenberg matrix assembled from the Schur parameters. For block extended Arnoldi we do not plot a projection error: in this implementation the algorithm only constructs the basis, so no projected matrix is formed during the run. This is only a choice for the present experiment: since extended Krylov spaces are rational Krylov spaces, one can also build the usual rational Krylov recurrence pair and recover a projected matrix from it; see~\cite{Guttel13}.

The results are consistent with the operation counts. At step $m$, the full orthogonalization methods compare the new block with all previously generated blocks, so the orthogonalization work is $O(s^2mn)$ and the total cost grows quadratically with $m$. The short-recurrence methods use only a fixed number of previously computed blocks; apart from the matrix-block multiplication, their orthogonalization work is $O(s^2n)$ per step, giving a linear growth with $m$. This behavior is visible in~\cref{fig:efficient-orthogonalization-timing}. The accuracy curves in~\cref{fig:efficient-orthogonalization-accuracy} remain close to roundoff level. As expected, the full orthogonalization baselines give the smallest loss of orthogonality, and standard block Arnoldi gives the smallest projection error. The short-recurrence methods are slightly less accurate, but still stable for this unitary problem while being substantially faster.

\subsection{\texorpdfstring{Repeated eigenvalues}{Repeated eigenvalues}}

\begin{figure}[ht]
\centering
\begin{tikzpicture}
\begin{semilogyaxis}[
    width=0.9\textwidth,
    height=0.48\textwidth,
    grid=both,
    xlabel={$m$},
    ylabel={$d_j^{(s)}(m)$},
    ymin=1e-4,
    legend style={font=\scriptsize, at={(0.5,1.03)}, anchor=south, inner sep=1.5pt, row sep=-1pt},
    legend columns=4,
    legend image post style={scale=0.75},
    tick label style={font=\footnotesize},
    label style={font=\small},
    unbounded coords=jump,
]
\addplot+[blue, solid, mark=x, line width=0.55pt, mark size=2.8pt] table[x=m,y=s1_closest1] {repeated_eigenvalue_ritz_distances_isometric_complex.dat};
\addlegendentry{$s=1$, $j=1$}
\addplot+[orange!80!black, solid, mark=x, line width=0.55pt, mark size=2.8pt] table[x=m,y=s1_closest2] {repeated_eigenvalue_ritz_distances_isometric_complex.dat};
\addlegendentry{$s=1$, $j=2$}
\addplot+[green!60!black, solid, mark=x, line width=0.55pt, mark size=2.8pt] table[x=m,y=s1_closest3] {repeated_eigenvalue_ritz_distances_isometric_complex.dat};
\addlegendentry{$s=1$, $j=3$}
\addplot+[red!75!black, solid, mark=x, line width=0.55pt, mark size=2.8pt] table[x=m,y=s1_closest4] {repeated_eigenvalue_ritz_distances_isometric_complex.dat};
\addlegendentry{$s=1$, $j=4$}
\addplot+[blue, dashed, mark=triangle, mark options={solid, fill=white}, line width=0.55pt, mark size=2.7pt] table[x=m,y=s4_closest1] {repeated_eigenvalue_ritz_distances_isometric_complex.dat};
\addlegendentry{$s=4$, $j=1$}
\addplot+[orange!80!black, dashed, mark=triangle, mark options={solid, fill=white}, line width=0.55pt, mark size=2.7pt] table[x=m,y=s4_closest2] {repeated_eigenvalue_ritz_distances_isometric_complex.dat};
\addlegendentry{$s=4$, $j=2$}
\addplot+[green!60!black, dashed, mark=triangle, mark options={solid, fill=white}, line width=0.55pt, mark size=2.7pt] table[x=m,y=s4_closest3] {repeated_eigenvalue_ritz_distances_isometric_complex.dat};
\addlegendentry{$s=4$, $j=3$}
\addplot+[red!75!black, dashed, mark=triangle, mark options={solid, fill=white}, line width=0.55pt, mark size=2.7pt] table[x=m,y=s4_closest4] {repeated_eigenvalue_ritz_distances_isometric_complex.dat};
\addlegendentry{$s=4$, $j=4$}
\addplot+[blue, only marks, mark=x, line width=0.9pt, mark size=3.6pt, forget plot] table[x=m,y=s1_closest1] {repeated_eigenvalue_ritz_distances_isometric_complex.dat};
\end{semilogyaxis}
\end{tikzpicture}
\caption{Distances from the repeated eigenvalue $1$ of the four closest Ritz values produced by block isometric Arnoldi, for block sizes $s=1$ and $s=4$.}
\label{fig:repeated-eigenvalue-ritz-distances}
\end{figure}
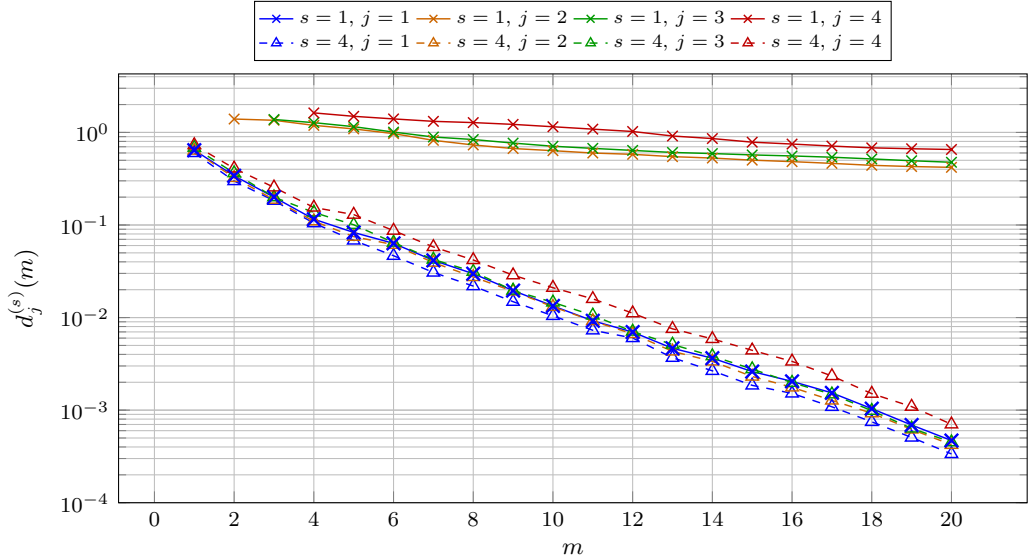

We now consider a unitary eigenvalue problem with a multiple eigenvalue. It is known that block Krylov projections are better suited for this setting; see, e.g.,~\cite{BK-multiple-eigenvalues}. The matrix has size $n=800$ and is a randomly generated complex unitary matrix with prescribed eigenvalues. Four eigenvalues are equal to $1$, while the remaining eigenvalues lie on the unit circle outside the arc corresponding to arguments in $(-0.35,0.35)$, so that the repeated eigenvalue is separated from the rest of the spectrum. The starting block is chosen with non-negligible components in the eigenspace associated with the repeated eigenvalue.

For $s=1$ and $s=4$, we run block isometric Arnoldi for $m=1,\dots,20$ and compute the Ritz values of the projected matrix $\mathcal H_m$. If $\theta_1,\dots,\theta_{ms}$ are the Ritz values, let
\begin{equation}
\label{eq:numerical-repeated-ritz-distance}
    d_j^{(s)}(m)
    =
    \text{the $j$th smallest value in }
    \bigl\{|\theta_i-1|:i=1,\dots,ms\bigr\},
    \qquad j=1,\dots,\min\{4,ms\}.
\end{equation}
The quantities $d_j^{(s)}(m)$ measure whether the projected problem contains one, two, three, or four Ritz values close to the repeated eigenvalue. We display only the block isometric Arnoldi data: indeed, by~\cref{prop:polynomial-extended-projected-similarity}, the projected matrix associated with block CMV Arnoldi would produce the same Ritz values in exact arithmetic.

The results are shown in~\cref{fig:repeated-eigenvalue-ritz-distances}. The plot shows the expected difference between the single-vector and block cases. For $s=1$, only the closest Ritz value approaches the repeated eigenvalue substantially, while the next three distances remain large. For $s=4$, all four distances decrease together, showing that the projected problem captures the multiplicity of the eigenvalue.

\section{Conclusions} 
  We have shown that, under the assumption of no deflation, there is an isometric correspondence between a block Krylov subspace and the space of matrix polynomials up to a prescribed degree. This yields a unified framework for the analysis and construction of orthonormal bases and recurrence relations. The same correspondence carries over to rational block Krylov subspaces and matrix-valued rational functions, which reduce to Laurent matrix polynomials in the extended Krylov setting.

We also proved that, when $A$ is normal, the inner product induced on matrix polynomials can be represented as an integral with respect to a discrete spectral measure, thereby extending a classical result for Hermitian matrices~\cite{GoMe10,LundPhD}. In the unitary case, the support of the spectral measure lies on the unit circle, which allows one to directly translate the Szeg\H{o} recurrence for orthogonal matrix polynomials, as well as CMV bases and CMV matrices for Laurent matrix polynomials, into efficient procedures for the orthogonalization of polynomial and extended block Krylov subspaces.

A limitation of the present analysis is the assumption of no deflation. Treating the deflated case requires a deeper understanding of degenerate inner products for matrix polynomials and remains an important direction for future work. Another natural continuation of this work is to exploit the connection developed here in order to derive alternative approaches to block Gaussian quadrature, especially in the unitary setting.

\section*{Funding}
The research was partially supported by the Research Council KU Leuven (Belgium), project C16/21/002 (Manifactor: Factor Analysis for Maps into Manifolds) and by the Fund for Scientific Research -- Flanders (Belgium), projects G0A9923N (Low rank tensor approximation techniques for up- and downdating of massive online time series clustering) and G0B0123N (Short recurrence relations for rational Krylov and orthogonal rational functions inspired by modified moments), and junior postdoctoral fellowship 12A1325N (Short Recurrences for Block Krylov Methods with Applications to Matrix Functions and Model Order Reduction) for the first author.

\bibliography{references}

@article {FrLuSz17,
    AUTHOR = {Frommer, Andreas and Lund, Kathryn and Szyld, Daniel B.},
     TITLE = {Block {K}rylov subspace methods for functions of matrices},
   JOURNAL = {Electron. Trans. Numer. Anal.},
  FJOURNAL = {Electronic Transactions on Numerical Analysis},
    VOLUME = {47},
      YEAR = {2017},
     PAGES = {100--126},
      ISSN = {1068-9613},
   MRCLASS = {65F60 (65F50)},
  MRNUMBER = {3707736},
MRREVIEWER = {A.\ Bultheel},
       DOI = {10.1553/etna\_vol47s100},
       URL = {https://doi.org/10.1553/etna_vol47s100},
}

@article {FrLuSz20,
    AUTHOR = {Frommer, Andreas and Lund, Kathryn and Szyld, Daniel B.},
     TITLE = {Block {K}rylov subspace methods for functions of matrices
              {II}: {M}odified block {FOM}},
   JOURNAL = {SIAM J. Matrix Anal. Appl.},
  FJOURNAL = {SIAM Journal on Matrix Analysis and Applications},
    VOLUME = {41},
      YEAR = {2020},
    NUMBER = {2},
     PAGES = {804--837},
      ISSN = {0895-4798,1095-7162},
   MRCLASS = {65F60 (65F50)},
  MRNUMBER = {4101368},
MRREVIEWER = {Yuriy\ Nikolaevich\ Belyayev},
       DOI = {10.1137/19M1255847},
       URL = {https://doi.org/10.1137/19M1255847},
}

@article {SinVanAs94,
    AUTHOR = {Sinap, Ann and Van Assche, Walter},
     TITLE = {Orthogonal matrix polynomials and applications},
   JOURNAL = {J. Comput. Appl. Math.},
  FJOURNAL = {Journal of Computational and Applied Mathematics},
    VOLUME = {66},
      YEAR = {1996},
    NUMBER = {1-2},
     PAGES = {27--52},
      ISSN = {0377-0427,1879-1778},
   MRCLASS = {42C05 (65D30)},
  MRNUMBER = {1393717},
MRREVIEWER = {Olav\ Nj\aa stad},
       DOI = {10.1016/0377-0427(95)00193-X},
       URL = {https://doi.org/10.1016/0377-0427(95)00193-X},
}

@book {LieStraKrylov,
    AUTHOR = {Liesen, J\"org and Strako\v{s}, Zden\v{e}k},
     TITLE = {{K}rylov subspace methods},
    SERIES = {Numerical Mathematics and Scientific Computation},
 PUBLISHER = {Oxford University Press},
      YEAR = {2013},
     PAGES = {xvi+391},
      ISBN = {978-0-19-965541-0},
   MRCLASS = {65F10 (65F15)},
  MRNUMBER = {3024841},
MRREVIEWER = {Melina\ A.\ Freitag},
}

@article {ZimDruSim25,
    AUTHOR = {Zimmerling, J\"orn and Druskin, Vladimir and Simoncini,
              Valeria},
     TITLE = {Monotonicity, bounds and acceleration of block {G}auss and
              {G}auss-{R}adau quadrature for computing {$B^T\phi(A)B$}},
   JOURNAL = {J. Sci. Comput.},
  FJOURNAL = {Journal of Scientific Computing},
    VOLUME = {103},
      YEAR = {2025},
    NUMBER = {1},
     PAGES = {Paper No. 5, 21},
      ISSN = {0885-7474,1573-7691},
   MRCLASS = {65F50 (65F10 65F60 65N22)},
  MRNUMBER = {4863914},
       DOI = {10.1007/s10915-025-02799-z},
       URL = {https://doi.org/10.1007/s10915-025-02799-z},
}

@article {Guttel13,
    AUTHOR = {G\"uttel, Stefan},
     TITLE = {Rational {K}rylov approximation of matrix functions: numerical
              methods and optimal pole selection},
   JOURNAL = {GAMM-Mitt.},
  FJOURNAL = {GAMM-Mitteilungen},
    VOLUME = {36},
      YEAR = {2013},
    NUMBER = {1},
     PAGES = {8--31},
      ISSN = {0936-7195,1522-2608},
   MRCLASS = {65F60},
  MRNUMBER = {3095912},
MRREVIEWER = {Volker\ Karl Richard Grimm},
       DOI = {10.1002/gamm.201310002},
       URL = {https://doi.org/10.1002/gamm.201310002},
}

@book {GoMe10,
    AUTHOR = {Golub, Gene H. and Meurant, G\'erard},
     TITLE = {Matrices, moments and quadrature with applications},
    SERIES = {Princeton Series in Applied Mathematics},
 PUBLISHER = {Princeton University Press},
      YEAR = {2010},
     PAGES = {xii+363},
      ISBN = {978-0-691-14341-5},
   MRCLASS = {65-02 (65D30)},
  MRNUMBER = {2582949},
MRREVIEWER = {G.\ A.\ Evans},
}

@phdthesis{LundPhD,
  author       = {Kathryn Lund},
  title        = {A new block {K}rylov subspace framework with applications to functions of matrices acting on multiple vectors},
  school       = {Temple University and Bergische Universit{\"a}t Wuppertal},
  year         = {2018},
  address      = {Philadelphia, PA and Wuppertal, Germany},
  url          = {https://www.katlund.com/research/thesis},
}

@article {DaPuSi08,
    AUTHOR = {Damanik, David and Pushnitski, Alexander and Simon, Barry},
     TITLE = {The analytic theory of matrix orthogonal polynomials},
   JOURNAL = {Surv. Approx. Theory},
  FJOURNAL = {Surveys in Approximation Theory},
    VOLUME = {4},
      YEAR = {2008},
     PAGES = {1--85},
      ISSN = {1555-578X},
   MRCLASS = {42C05 (30C10 47B36)},
  MRNUMBER = {2379691},
}

@article {SimonCMV07,
    AUTHOR = {Simon, Barry},
     TITLE = {{CMV} matrices: five years after},
   JOURNAL = {J. Comput. Appl. Math.},
  FJOURNAL = {Journal of Computational and Applied Mathematics},
    VOLUME = {208},
      YEAR = {2007},
    NUMBER = {1},
     PAGES = {120--154},
      ISSN = {0377-0427,1879-1778},
   MRCLASS = {47B36 (15A52)},
  MRNUMBER = {2347741},
MRREVIEWER = {Petru\ A.\ Cojuhari},
       DOI = {10.1016/j.cam.2006.10.033},
       URL = {https://doi.org/10.1016/j.cam.2006.10.033},
}

@article {Watkins93,
    AUTHOR = {Watkins, David S.},
     TITLE = {Some perspectives on the eigenvalue problem},
   JOURNAL = {SIAM Rev.},
  FJOURNAL = {SIAM Review. A Publication of the Society for Industrial and
              Applied Mathematics},
    VOLUME = {35},
      YEAR = {1993},
    NUMBER = {3},
     PAGES = {430--471},
      ISSN = {1095-7200},
   MRCLASS = {65F15 (65D32 65F25)},
  MRNUMBER = {1234638},
MRREVIEWER = {E.\ Kreyszig},
       DOI = {10.1137/1035090},
       URL = {https://doi.org/10.1137/1035090},
}

@article {ElsGutt20,
    AUTHOR = {Elsworth, Steven and G\"uttel, Stefan},
     TITLE = {The block rational {A}rnoldi method},
   JOURNAL = {SIAM J. Matrix Anal. Appl.},
  FJOURNAL = {SIAM Journal on Matrix Analysis and Applications},
    VOLUME = {41},
      YEAR = {2020},
    NUMBER = {2},
     PAGES = {365--388},
      ISSN = {0895-4798,1095-7162},
   MRCLASS = {65F15 (65F25 65F50)},
  MRNUMBER = {4082482},
MRREVIEWER = {Jajati\ Keshari\ Sahoo},
       DOI = {10.1137/19M1245505},
       URL = {https://doi.org/10.1137/19M1245505},
}

@article {Gragg93isometric,
    AUTHOR = {Gragg, William B.},
     TITLE = {Positive definite {T}oeplitz matrices, the {A}rnoldi process
              for isometric operators, and {G}aussian quadrature on the unit
              circle},
   JOURNAL = {J. Comput. Appl. Math.},
  FJOURNAL = {Journal of Computational and Applied Mathematics},
    VOLUME = {46},
      YEAR = {1993},
    NUMBER = {1-2},
     PAGES = {183--198},
      ISSN = {0377-0427,1879-1778},
   MRCLASS = {65F15 (62M20 65D32)},
  MRNUMBER = {1222480},
       DOI = {10.1016/0377-0427(93)90294-L},
       URL = {https://doi.org/10.1016/0377-0427(93)90294-L},
}

@book {Simon-book05,
    AUTHOR = {Simon, Barry},
     TITLE = {Orthogonal polynomials on the unit circle. {P}art 1},
    SERIES = {American Mathematical Society Colloquium Publications},
    VOLUME = {54, Part 1},
 PUBLISHER = {American Mathematical Society},
      YEAR = {2005},
     PAGES = {xxvi+466},
      ISBN = {0-8218-3446-0},
   MRCLASS = {42-02 (30C85 33C45 42C05 47B36 47N50)},
  MRNUMBER = {2105088},
MRREVIEWER = {P.\ L.\ Duren},
       DOI = {10.1090/coll054.1},
       URL = {https://doi.org/10.1090/coll054.1},
}

@article {CarsonLundetal22,
    AUTHOR = {Carson, Erin and Lund, Kathryn and Rozlo\v{z}n\'ik, Miroslav
              and Thomas, Stephen},
     TITLE = {Block {G}ram-{S}chmidt algorithms and their stability
              properties},
   JOURNAL = {Linear Algebra Appl.},
  FJOURNAL = {Linear Algebra and its Applications},
    VOLUME = {638},
      YEAR = {2022},
     PAGES = {150--195},
      ISSN = {0024-3795,1873-1856},
   MRCLASS = {65F25 (15A23 65F05 65F10)},
  MRNUMBER = {4360388},
MRREVIEWER = {Changqing\ Xu},
       DOI = {10.1016/j.laa.2021.12.017},
       URL = {https://doi.org/10.1016/j.laa.2021.12.017},
}

@article {MeurantLanczos06,
    AUTHOR = {Meurant, G\'erard and Strako\v{s}, Zden\v{e}k},
     TITLE = {The {L}anczos and conjugate gradient algorithms in finite
              precision arithmetic},
   JOURNAL = {Acta Numer.},
  FJOURNAL = {Acta Numerica},
    VOLUME = {15},
      YEAR = {2006},
     PAGES = {471--542},
      ISSN = {0962-4929,1474-0508},
      ISBN = {0-521-86815-7},
   MRCLASS = {65F15 (65F10 65G50)},
  MRNUMBER = {2269746},
MRREVIEWER = {A.\ Bultheel},
       DOI = {10.1017/S096249290626001X},
       URL = {https://doi.org/10.1017/S096249290626001X},
}

@book {SzegoBook75,
    AUTHOR = {Szeg\H{o}, G\'abor},
     TITLE = {Orthogonal polynomials},
    SERIES = {American Mathematical Society Colloquium Publications},
    VOLUME = {23},
   EDITION = {Fourth},
 PUBLISHER = {American Mathematical Society},
      YEAR = {1975},
     PAGES = {xiii+432},
   MRCLASS = {42A52 (33A65)},
  MRNUMBER = {372517},
}

@article {Simoncini95,
    AUTHOR = {Simoncini, V. and Gallopoulos, E.},
     TITLE = {An iterative method for nonsymmetric systems with multiple
              right-hand sides},
   JOURNAL = {SIAM J. Sci. Comput.},
  FJOURNAL = {SIAM Journal on Scientific Computing},
    VOLUME = {16},
      YEAR = {1995},
    NUMBER = {4},
     PAGES = {917--933},
      ISSN = {1064-8275},
   MRCLASS = {65F10 (15A06 65Y20)},
  MRNUMBER = {1335897},
MRREVIEWER = {R.\ E.\ Funderlic},
       DOI = {10.1137/0916053},
       URL = {https://doi.org/10.1137/0916053},
}

@article {OLeary80,
    AUTHOR = {O'Leary, Dianne P.},
     TITLE = {The block conjugate gradient algorithm and related methods},
   JOURNAL = {Linear Algebra Appl.},
  FJOURNAL = {Linear Algebra and its Applications},
    VOLUME = {29},
      YEAR = {1980},
     PAGES = {293--322},
      ISSN = {0024-3795,1873-1856},
   MRCLASS = {65F10 (90C20)},
  MRNUMBER = {562766},
MRREVIEWER = {Jan\ Z\'itko},
       DOI = {10.1016/0024-3795(80)90247-5},
       URL = {https://doi.org/10.1016/0024-3795(80)90247-5},
}

@article {blockKrylovSylvester00,
    AUTHOR = {El Guennouni, A. and Jbilou, K. and Riquet, A. J.},
     TITLE = {Block {K}rylov subspace methods for solving large {S}ylvester
              equations},
   JOURNAL = {Numer. Algorithms},
  FJOURNAL = {Numerical Algorithms},
    VOLUME = {29},
      YEAR = {2002},
    NUMBER = {1-3},
     PAGES = {75--96},
      ISSN = {1017-1398,1572-9265},
   MRCLASS = {65F10},
  MRNUMBER = {1896947},
MRREVIEWER = {Jian-Ping\ Zhu},
       DOI = {10.1023/A:1014807923223},
       URL = {https://doi.org/10.1023/A:1014807923223},
}

@book {Antoulas05,
    AUTHOR = {Antoulas, Athanasios C.},
     TITLE = {Approximation of large-scale dynamical systems},
    SERIES = {Advances in Design and Control},
    VOLUME = {6},
 PUBLISHER = {Society for Industrial and Applied Mathematics},
      YEAR = {2005},
     PAGES = {xxvi+479},
      ISBN = {0-89871-529-6},
   MRCLASS = {93-02 (65P99 93B25 93B40 93C05)},
  MRNUMBER = {2155615},
MRREVIEWER = {Petko\ Petkov},
       DOI = {10.1137/1.9780898718713},
       URL = {https://doi.org/10.1137/1.9780898718713},
}

@book {Parlett98,
    AUTHOR = {Parlett, Beresford N.},
     TITLE = {The symmetric eigenvalue problem},
    SERIES = {Classics in Applied Mathematics},
    VOLUME = {20},
 PUBLISHER = {Society for Industrial and Applied Mathematics},
      YEAR = {1998},
     PAGES = {xxiv+398},
      ISBN = {0-89871-402-8},
   MRCLASS = {65F15 (15A18)},
  MRNUMBER = {1490034},
MRREVIEWER = {F.\ Szidarovszky},
       DOI = {10.1137/1.9781611971163},
       URL = {https://doi.org/10.1137/1.9781611971163},
}

@article {BlockGradeGutknecht09,
    AUTHOR = {Gutknecht, Martin H. and Schmelzer, Thomas},
     TITLE = {The block grade of a block {K}rylov space},
   JOURNAL = {Linear Algebra Appl.},
  FJOURNAL = {Linear Algebra and its Applications},
    VOLUME = {430},
      YEAR = {2009},
    NUMBER = {1},
     PAGES = {174--185},
      ISSN = {0024-3795,1873-1856},
   MRCLASS = {65F10},
  MRNUMBER = {2460508},
MRREVIEWER = {Valeria\ Simoncini},
       DOI = {10.1016/j.laa.2008.07.008},
       URL = {https://doi.org/10.1016/j.laa.2008.07.008},
}

@article {HelKuiVanba05,
    AUTHOR = {Helsen, S. and Kuijlaars, A. B. J. and Van Barel, M.},
     TITLE = {Convergence of the isometric {A}rnoldi process},
   JOURNAL = {SIAM J. Matrix Anal. Appl.},
  FJOURNAL = {SIAM Journal on Matrix Analysis and Applications},
    VOLUME = {26},
      YEAR = {2005},
    NUMBER = {3},
     PAGES = {782--809},
      ISSN = {0895-4798,1095-7162},
   MRCLASS = {65F15 (31A05 31A15)},
  MRNUMBER = {2137484},
       DOI = {10.1137/S0895479803438201},
       URL = {https://doi.org/10.1137/S0895479803438201},
}

@article {EmbreeFloquet24,
    AUTHOR = {Damanik, David and Embree, Mark and Fillman, Jake},
     TITLE = {Gap labels for zeros of the partition function of the 1{D}
              {I}sing model via the {S}chwartzman homomorphism},
   JOURNAL = {Indag. Math. (N.S.)},
  FJOURNAL = {Koninklijke Nederlandse Akademie van Wetenschappen.
              Indagationes Mathematicae. New Series},
    VOLUME = {35},
      YEAR = {2024},
    NUMBER = {5},
     PAGES = {813--836},
      ISSN = {0019-3577,1872-6100},
   MRCLASS = {82B44 (15A18 37A25 47B36)},
  MRNUMBER = {4793392},
MRREVIEWER = {Qiang\ Wu},
       DOI = {10.1016/j.indag.2023.05.004},
       URL = {https://doi.org/10.1016/j.indag.2023.05.004},
}

@article {BK-multiple-eigenvalues,
    AUTHOR = {Sadkane, Miloud},
     TITLE = {Block-{A}rnoldi and {D}avidson methods for unsymmetric large
              eigenvalue problems},
   JOURNAL = {Numer. Math.},
  FJOURNAL = {Numerische Mathematik},
    VOLUME = {64},
      YEAR = {1993},
    NUMBER = {2},
     PAGES = {195--211},
      ISSN = {0029-599X,0945-3245},
   MRCLASS = {65F50 (65F15)},
  MRNUMBER = {1199285},
       DOI = {10.1007/BF01388687},
       URL = {https://doi.org/10.1007/BF01388687},
}

@article {extended-block-arnoldi,
    AUTHOR = {Druskin, Vladimir and Knizhnerman, Leonid},
     TITLE = {Extended {K}rylov subspaces: approximation of the matrix
              square root and related functions},
   JOURNAL = {SIAM J. Matrix Anal. Appl.},
  FJOURNAL = {SIAM Journal on Matrix Analysis and Applications},
    VOLUME = {19},
      YEAR = {1998},
    NUMBER = {3},
     PAGES = {755--771},
      ISSN = {0895-4798,1095-7162},
   MRCLASS = {65F25 (65N99)},
  MRNUMBER = {1616584},
MRREVIEWER = {Larisa\ V.\ Maslovskaya},
       DOI = {10.1137/S0895479895292400},
       URL = {https://doi.org/10.1137/S0895479895292400},
}
\bibliographystyle{abbrv}

\end{document}